\documentclass[10pt,dvipsnames]{amsart}
\usepackage{hyperref}
\usepackage{amssymb, amscd, amsmath, amsfonts, amsthm}
\usepackage{stmaryrd}
\usepackage{verbatim}

\def\unprotectedboldentry#1{\textcolor{Red}{\textbf{#1}}}
\def\boldentry{\protect\unprotectedboldentry}
\usepackage{tikz}
\usetikzlibrary{calc}
\usepackage{ifthen}
\newcommand{\tikztableauinternal}[1]{
    \def\newtableau{#1}
    \coordinate (x) at (-0.5,0.5);
    \coordinate (y) at (-0.5,0.5);
    \foreach \row in \newtableau {
        \coordinate (x) at ($(x)-(0,1)$);
        \coordinate (y) at (x);
        \foreach \entry in \row {
            \ifthenelse{\equal{\entry}{X}}
               {
                \node (y) at ($(y) + (1,0)$) {};
                \fill[color=gray!10] ($(y)-(0.5,0.5)$) rectangle +(1,1);
                \draw[color=gray] ($(y)-(0.5,0.5)$) rectangle +(1,1);
               }
               {
                \ifthenelse{\equal{\entry}{\boldentry X}}
                   {
                    \node (y) at ($(y) + (1,0)$) {};
                    \fill[color=gray] ($(y)-(0.5,0.5)$) rectangle +(1,1);
                    \draw ($(y)-(0.5,0.5)$) rectangle +(1,1);
                   }
                   {
                    \node (y) at ($(y) + (1,0)$) {\entry};
                    \draw ($(y)-(0.5,0.5)$) rectangle +(1,1);
                   }
               }
            }
        }
}

\newcommand{\tikztableau}[2][scale=0.6,every node/.style={font=\small}]{
    \begin{array}{c}
    \begin{tikzpicture}[#1]
        \tikztableauinternal{#2}
    \end{tikzpicture}
    \end{array}}
\newcommand{\tikztableausmall}[1]{\tikztableau[scale=0.45,every node/.style={font=\rm\small}]{#1}}

\def\uu{\mathbf u}
\def\hh{\mathbf h}
\def\bb{\mathbf b}
\def\AA{\mathbb{A}}
\def\BB{\mathbb{B}}
\def\XX{\mathfrak{X}}

\def\downgraph{\mathcal G^{\downarrow}}
\def\ascomp{\operatorname{ascomp}}

\newcommand\kschur[1][k]{s^{(#1)}}
\newcommand\dualkschur{\widetilde F}
\newcommand\nckschur[1][k]{\mathfrak{s}^{(#1)}}

\def\strongstrip{\rightarrowtriangle}
\def\size{\operatorname{size}}

\def\kBounded{\mathcal B^{(k)}}
\def\Langle{\left\langle}
\def\Rangle{\right\rangle}

\def\mod{\operatorname{\rm{mod}}}

\newtheorem{Theorem}{Theorem}[section]
\newtheorem{Proposition}[Theorem]{Proposition}
\newtheorem{Corollary}[Theorem]{Corollary}
\newtheorem{Lemma}[Theorem]{Lemma}
\newtheorem{Example}[Theorem]{Example}
\newtheorem{Remark}[Theorem]{Remark}
\newtheorem{Definition}[Theorem]{Definition}

\allowdisplaybreaks
\begin{document}

\title[Combinatorial expansions for non-commutative $k$-Schur functions]{Combinatorial expansions for families of non-commutative $k$-Schur functions}
\author{Chris Berg \and Franco Saliola \and Luis Serrano}
\date{\today}

\begin{abstract}
We apply down operators in the affine nilCoxeter algebra to yield explicit combinatorial expansions for certain families of non-commutative $k$-Schur functions. This yields a combinatorial interpretation for a new family of $k$-Littlewood-Richardson coefficients.
\end{abstract}

\maketitle

\section{Introduction}

The $k$-Schur functions of Lapointe, Lascoux and Morse \cite{LLM03} first arose
in the study of Macdonald polynomials. Since then, their study has flourished;
see for instance \cite{LM03, LM05, LM07, LS07, LLMS10, Lam10} and the
references therein. This is due, in part, to an important geometric
interpretation of the Hopf algebra $\Lambda_{(k)}$ of $k$-Schur functions and
its dual Hopf algebra $\Lambda^{(k)}$: these algebras are isomorphic to the
homology and cohomology of the affine Grassmannian in type~A \cite{Lam08}.
Under this isomorphism, the $k$-Schur functions map to the Schubert basis of
the homology and the dual $k$-Schur functions (also called the affine Schur
functions) map to the Schubert basis of the cohomology.

An important problem in this field is to find a
$k$-Littlewood--Richardson rule, namely, a combinatorial interpretation for the
(nonnegative) coefficients in the expansion
\begin{equation}\label{eq:klr}
\kschur_{\mu} \kschur_{\nu} = \sum_{\lambda} c^{\lambda, (k)}_{\mu, \nu} \kschur_{\lambda}.
\end{equation}
The $c^{\lambda, (k)}_{\mu, \nu}$ are called the
$k$-Littlewood--Richardson-coefficients, and are of high relevance in
combinatorics and geometry. It was proved by Lapointe and Morse \cite{LM08}
that special cases of these coefficients yield the 3-point Gromov--Witten
invariants. These invariants are the structure constants of the quantum
cohomology of the Grassmanian; they count the number of rational curves of
a fixed degree in the Grassmannian. Benedetti, Bergeron and Zabrocki \cite{BBZ} and Morse and Schilling \cite{MS} have both recently discovered other ways to compute these numbers in special cases. The 3-point Gromov-Witten invariants of
flag varieties are also $k$-Littlewood--Richardson coefficients; see
\cite{LSGW,leung_gromov-witten_2012}.

As an approach to finding the $k$-Littlewood--Richardson coefficients, Lam
\cite{Lam06} identified $\Lambda_{(k)}$ with the Fomin--Stanley subalgebra
$\BB$ of the affine nilCoxeter algebra $\AA$ of the affine symmetric group $W$.
Specifically, he constructed a family of elements $\nckschur_{\lambda}\in\BB$
that map under this isomorphism to the $k$-Schur functions
$\kschur_{\lambda}$. Furthermore, he proved \cite[Proposition 42]{Lam06} that
finding the $k$-Littlewood--Richardson rule is equivalent to finding the
expansion of the $\nckschur_{\lambda}$ in the ``standard basis'' $\uu_w$ of
$\AA$. Explicitly, he proved that the coefficients in \eqref{eq:klr} appear as
coefficients in the expansion
\begin{equation}\label{eq:expansion}
\nckschur_\lambda = \sum_{w\in W} d_\lambda^w \uu_w.
\end{equation}

The goal of this article is to obtain expansions of the $\nckschur_\lambda$ in $\AA$ for certain families of shapes. In Sections \ref{sec:comb_desc} and \ref{sec:expansions}, we use the Pieri
operators to find explicit combinatorial interpretations for the coefficients
in (\ref{eq:expansion}) (and thus, for the coefficients in (\ref{eq:klr}) as
well) for certain families of shapes. We determine the expansion in $\AA$ of
$\nckschur_\lambda$ for the shapes $\lambda = (c^{k-c},c-i)$, for all $i \geq
0$. These correspond to rectangles of $c$ rows and $k+1-c$ columns with a
horizontal strip of length $i$ removed. See Theorem \ref{thm:rectangle-strip} for the precise statement. The case $i=0$ was first established in
\cite{BBTZ11}, the case $i=1$ was first established in \cite{BSS11}.

This manuscript is the third instalment in the `down operators' saga by the authors. The interested reader would do well to first consult \cite{BSS11} and \cite{BSS12}. 

\subsection{Acknowledgements}
We would like to thank Nantel Bergeron, Tom Denton, Thomas Lam, Steven Pon, Jennifer Morse, Anne
Schilling, and Mike Zabrocki for helpful discussions throughout the
course of this research.

This research was facilitated by computer exploration using the open-source
mathematical software system \texttt{Sage}~\cite{sage} and its algebraic
combinatorics features developed by the \texttt{Sage-Combinat} community
\cite{sage-combinat}.

\section{Background and Notation}

\subsection{Affine symmetric group}
Fix a positive integer $k$. Let $W$ denote the affine symmetric group with
simple generators $s_0, s_1, \ldots, s_k$. There is an
interpretation of $W$ as the group of permutations $w: \mathbb Z \to \mathbb Z$
satisfying $w(i+k+1) = w(i) + k+1$ for all $i\in\mathbb Z$ and $\sum_{i=1}^{k+1}
w(i)=\sum_{i=1}^{k+1} i$. Let $t_{i,j}$ be the element of $W$ that interchanges the
integers $i$ and $j$ and fixes all integers not congruent to $i$ or $j$ modulo
$k+1$.

Let $W_0$ denote the subgroup of $W$ generated by $s_1, \dots, s_k$ and let
$W^0$ denote the set of minimal length coset representatives of $W/W_0$.
Elements of $W^0$ are called \emph{affine Grassmannian elements} or
\emph{$0$-Grassmannian elements}. There are bijections between $0$-Grassmannian
elements, $k$-bounded partitions, and $(k+1)$-cores. We will not review these
here, but refer the reader to \cite{LM05}. For a $k$-bounded partition $\lambda$,
we let $w_\lambda$ denote the corresponding element of $W^0$. Let $\kBounded$
denote the set of $k$-bounded partitions.

\subsection{Affine nilCoxeter algebra}
Let $\AA$ denote the \emph{affine nilCoxeter algebra} of $W$: this is the
algebra generated by $\uu_0, \uu_1, \dots, \uu_k$ with relations:
\begin{gather*}
    \uu_i^2 = 0 \text{ for all $i$}; \\
    \uu_i \uu_{i+1} \uu_i = \uu_{i+1} \uu_i \uu_{i+1}
        \text{ with $i+1$ taken modulo $k+1$}; \\
    \uu_i \uu_j = \uu_j \uu_i
        \text{ if $i - j \neq \pm1$ modulo $k+1$.}
\end{gather*}
It follows that a basis of $\AA$ is given by the elements
$\uu_w = \uu_{s_{i_1}}\uu_{s_{i_2}}\cdots\uu_{s_{i_l}}$,
where $w = s_{i_1}s_{i_2}\cdots s_{i_l}$ is a reduced word for $w \in W$.
We define an inner product on $\AA$ by
$\langle \uu_v, \uu_w \rangle_\AA = \delta_{u,v}$.

\subsection{Affine Fomin--Stanley subalgebra}
An element $w \in W$ is said to be \emph{cyclically decreasing} if there exists
a reduced factorization $s_{i_1}\cdots s_{i_j}$ of $w$ satisfying: each letter
occurs at most once; and, for all $m$, if $s_m$ and $s_{m+1}$ both appear in
the reduced factorization, then $s_{m+1}$ precedes $s_m$. If $D \subsetneq \{0,
1, \dots, k\}$, then there is a unique cyclically decreasing element $w_D$ with
letters $\{s_d : d \in D\}$. Let $\uu_D = \uu_{w_D}$ denote the corresponding
basis element of $\AA$. For $i \in \{0, 1, \dots, k\}$, let
\begin{align*}
\hh_i = \sum_{\substack{D \subset I \\ |D| = i}} \uu_D \in \AA.
\end{align*}
By a result of Thomas Lam \cite{Lam06}, the elements $\{\hh_i\}_{i \leq k}$ commute
and freely generate a subalgebra $\BB$ of $\AA$ called the \emph{affine
Fomin--Stanley subalgebra}. The elements $\hh_\lambda = \hh_{\lambda_1} \dots
\hh_{\lambda_t}$, for all $k$-bounded partitions $\lambda = (\lambda_1, \dots,
\lambda_t)$, form a basis of $\BB$.

\subsection{Symmetric functions}
\label{kspace}

Let $\Lambda$ denote the ring of symmetric functions. For a partition
$\lambda$, we let $m_\lambda$, $h_\lambda$ and
$s_\lambda$ denote the monomial, homogeneous and Schur
symmetric function, respectively, indexed by $\lambda$.

Let $\Lambda_{(k)}$ denote the subalgebra of $\Lambda$ generated by $h_0$,
$h_1$, $\dots$, $h_k$. The elements $h_\lambda$ with $\lambda_1 \leq k$ form a
basis of $\Lambda_{(k)}$. Let $\Lambda^{(k)} = \Lambda/I_k$ denote the quotient
of $\Lambda$ by the ideal $I_k$ generated by $m_\lambda$ with $\lambda_1 > k$.
The equivalence classes in $\Lambda^{(k)}$ of the elements $m_\lambda$ with
$\lambda_1 \leq k$ form a basis of $\Lambda^{(k)}$. 

The \emph{Hall inner product} of symmetric functions is defined by
$$
 \langle h_\lambda, m_\mu \rangle_\Lambda = 
 \langle m_\lambda, h_\mu \rangle_\Lambda = 
 \langle s_\lambda, s_\mu \rangle_\Lambda = 
 \delta_{\lambda,\mu}.
$$
Observe that every element of the ideal $I_k$ is orthogonal to every element of
$\Lambda_{(k)}$ with respect to this inner product. Hence, it induces a
pairing $\langle \cdot, \cdot \rangle$ between $\Lambda_{(k)}$ and
$\Lambda^{(k)}$. In particular, $\langle f, g \rangle = \langle f, \widetilde g
\rangle$ for $f \in \Lambda_{(k)}$, $g \in \Lambda^{(k)}$ and any preimage
$\widetilde g$ of $g$ under the quotient map $\Lambda \to \Lambda^{(k)}$.
For an element $f$ in $\Lambda^{(k)}$, we let $f^\perp : \Lambda_{(k)} \to
\Lambda_{(k)}$ denote the linear operator that is adjoint to multiplication by $f$
with respect to $\Langle \cdot, \cdot \Rangle$.

\subsection{Affine Schur functions}
The affine Schur functions form a distinguished basis of $\Lambda^{(k)}$.
For $w \in W$, the \emph{affine Stanley symmetric function} is defined as
\begin{gather*}
    \dualkschur_w = \sum_{\lambda \in \kBounded}
    \big\langle \hh_\lambda, \uu_w \big\rangle_\AA \, m_\lambda.
\end{gather*}
These functions are elements of $\Lambda^{(k)}$, but they are not linearly
independent. For a $k$-bounded partition $\lambda$, let $\dualkschur_\lambda =
\dualkschur_{w_\lambda}$, where $w_\lambda$ denotes the $0$-Grassmannian
element corresponding to $\lambda$. The functions $\dualkschur_\lambda$ are
called \emph{affine Schur functions} (or \emph{dual $k$-Schur functions}) and
they form a basis of $\Lambda^{(k)}$. See for instance \cite{Lam06, LM08}.

\subsection{$k$-Schur functions}
\label{ss:kschurs}
The $k$-Schur functions are a distinguished basis of $\Lambda_{(k)}$.
They are defined as the duals of the affine Schur functions with respect
to the inner product $\langle \cdot, \cdot \rangle$ on $\Lambda_{(k)} \times
\Lambda^{(k)}$. That is, they satisfy
$\langle \kschur_\lambda, \dualkschur_\mu \rangle = \delta_{\lambda,\mu}$
for all $k$-bounded partitions $\lambda$ and $\mu$.
Equivalently, they are uniquely defined by the \emph{$k$-Pieri rule}:
\begin{align*}
    h_i \kschur_\lambda = \sum \kschur_\nu
\end{align*}
where the sum ranges over all $k$-bounded partitions $\nu$ such that
$w_\nu w_\lambda^{-1}$ is cyclically decreasing of length $i$.
It follows from duality that
\begin{align*}
    h_\mu = \sum_{\lambda\in\kBounded}
    \big\langle \hh_\lambda, \uu_{w_\mu} \big\rangle_\AA
    \, \kschur_\lambda.
\end{align*}

\subsection{Non-commutative $k$-Schur functions}
\label{ss:noncommkschurs}
The algebras $\Lambda_{(k)}$ and $\BB$ are isomorphic with
isomorphism given by $h_\lambda \mapsto \hh_\lambda$.
We denote by $\nckschur_\lambda$ the image of the
$k$-Schur function $\kschur_\lambda$ under this isomorphism.
In the literature, $\nckschur_\lambda$ is called
a \emph{non-commutative $k$-Schur function}.
They have the following expansion \cite[Proposition~42]{Lam06}:
\begin{gather}
    \label{eq:kschurexpansioninA}
    \nckschur_\lambda = \sum_{w \in W}
    \Langle \kschur_\lambda, \dualkschur_w \Rangle
    \, \uu_w.
\end{gather}
That is, the coefficient of $\uu_w$ in $\nckschur_\lambda$ is equal
to the coefficient of $\dualkschur_\lambda$ in $\dualkschur_w$:
\begin{gather}
    \label{eq:relationbetweenpairings}
    \Langle \kschur_\lambda, \dualkschur_w \Rangle
    =
    \Langle \nckschur_\lambda, \uu_w \Rangle_\AA
\end{gather}
and so
\begin{gather*}
    \dualkschur_w = \sum_{\lambda\in\kBounded}
    \Langle \nckschur_\lambda, \uu_w \Rangle_\AA \dualkschur_\lambda.
\end{gather*}
Consequently, $\nckschur_\lambda$ contains exactly one term $\uu_w$ with $w \in
W^0$ and its coefficient is $1$. Furthermore, if $\sum_{w} c_w \uu_w$ is known
to lie in $\BB$, then
$\sum_{w} c_w \uu_w = \sum_{\lambda} c_{w_\lambda} \nckschur_{\lambda}$.

\section{Down operators}

In this section, we recall definition of the down operators, $D_J$, on the affine nilCoxeter algebra $\AA$ defined in \cite{BSS12}. The definitions are dependent upon the combinatorics introduced by Lam, Lapointe, Morse, and Shimozono in \cite{LLMS10}.

We define an edge-labelled oriented graph $\downgraph$, the
\emph{marked strong order} graph, with vertex set $W$: for every instance (if any) of $i \leq 0 < j$ such that $y \, t_{i,j} = x$, and $\ell(x) = \ell(y)+1$, there is an edge from $x$ to $y$ labelled by $y(j)=x(i)$. Note that $\downgraph$ allows multiple edges between two vertices.


\begin{Example} $(k=2)$
There are two edges from $x = s_0s_1s_2s_0$ to $y = s_1s_2s_0$ since $y^{-1} x$
can be written as $t_{i,j}$ with $i \leq 0 < j$ in two ways: $y^{-1} x =
t_{-4,1} = t_{-1,4}$. These edges are labelled by $y(1) = -2$ and $y(4) = 1$.
See Figure \ref{fig:downgraph}.
\end{Example}

\begin{Remark}
\cite{LLMS10} defined a similar graph except that they oriented their edges in
the opposite direction and labelled the edges by the pair $(i,j)$:
they write $y \buildrel{i,j}\over{\longrightarrow} x$ whereas we write $x
\buildrel{y(j)}\over{\longrightarrow} y$; and they call our label $y(j)$ the
\emph{marking} of the edge. See also Remark \ref{remark:otherresidues}.
\end{Remark}

\begin{figure}[htb]
\begin{center}
\begin{tikzpicture}[scale=0.60,>=latex,line join=bevel,]
  \node (u0*u1*u2) at (74.359bp,212bp) [draw,draw=none] {$s_{0}s_{1}s_{2}$};
  \node (u2*u1*u0) at (368.36bp,212bp) [draw,draw=none] {$s_{2}s_{1}s_{0}$};
  \node (u2*u1) at (418.36bp,144bp) [draw,draw=none] {$s_{2}s_{1}$};
  \node (u1*u2*u1*u0) at (238.36bp,280bp) [draw,draw=none] {$s_{1}s_{2}s_{1}s_{0}$};
  \node (u1*u2) at (104.36bp,144bp) [draw,draw=none] {$s_{1}s_{2}$};
  \node (u1*u0) at (327.36bp,144bp) [draw,draw=none] {$s_{1}s_{0}$};
  \node (u2) at (118.36bp,76bp) [draw,draw=none] {$s_{2}$};
  \node (u1) at (386.36bp,76bp) [draw,draw=none] {$s_{1}$};
  \node (u0*u1*u0) at (484.36bp,212bp) [draw,draw=none] {$s_{0}s_{1}s_{0}$};
  \node (u1*u2*u0) at (156.36bp,212bp) [draw,draw=none] {$s_{1}s_{2}s_{0}$};
  \node (u0) at (272.36bp,76bp) [draw,draw=none] {$s_{0}$};
  \node (1) at (272.36bp,7bp) [draw,draw=none] {$1$};
  \node (u0*u1*u2*u0) at (87.359bp,280bp) [draw,draw=none] {$s_{0}s_{1}s_{2}s_{0}$};
  \node (u2*u0) at (197.36bp,144bp) [draw,draw=none] {$s_{2}s_{0}$};
  \node (u0*u1) at (464.36bp,144bp) [draw,draw=none] {$s_{0}s_{1}$};
  \node (u0*u2) at (38.359bp,144bp) [draw,draw=none] {$s_{0}s_{2}$};
  \node (u0*u2*u1*u0) at (470.36bp,280bp) [draw,draw=none] {$s_{0}s_{2}s_{1}s_{0}$};
  \node (u1*u2*u1) at (238.36bp,212bp) [draw,draw=none] {$s_{1}s_{2}s_{1}$};
  \node (u0*u2*u1) at (571.36bp,212bp) [draw,draw=none] {$s_{0}s_{2}s_{1}$};
  \node (u0*u2*u0) at (19.359bp,212bp) [draw,draw=none] {$s_{0}s_{2}s_{0}$};
  \draw [black,->] (u2*u0) ..controls (177.61bp,127bp) and (150.43bp,103.6bp)  .. (u2);
  \definecolor{strokecol}{rgb}{0.0,0.0,0.0};
  \pgfsetstrokecolor{strokecol}
  \tikzstyle{every node}=[font=\footnotesize]
  \draw (175.36bp,110bp) node {$1$};
  \draw [black,->] (u1*u2*u0) ..controls (186.42bp,201.68bp) and (202.25bp,195.37bp)  .. (215.36bp,188bp) .. controls (228.38bp,180.68bp) and (228.95bp,174.58bp)  .. (242.36bp,168bp) .. controls (262.19bp,158.26bp) and (286.63bp,151.88bp)  .. (u1*u0);
  \draw (251.36bp,178bp) node {$0$};
  \draw [black,->] (u0*u1*u2*u0) ..controls (75.947bp,269.17bp) and (70.733bp,262.83bp)  .. (68.359bp,256bp) .. controls (65.223bp,246.99bp) and (66.522bp,236.41bp)  .. (u0*u1*u2);
  \draw (77.359bp,246bp) node {$2$};
  \draw [black,->] (u1*u0) ..controls (342.16bp,126.94bp) and (361.5bp,104.65bp)  .. (u1);
  \draw (372.36bp,110bp) node {$2$};
  \draw [black,->] (u0*u2*u1) ..controls (549.49bp,196.54bp) and (522.98bp,178.91bp)  .. (498.36bp,168bp) .. controls (488.13bp,163.47bp) and (461.18bp,155.68bp)  .. (u2*u1);
  \draw (543.36bp,178bp) node {$1$};
  \draw [black,->] (u2*u1*u0) ..controls (342.34bp,202.2bp) and (333.24bp,196.4bp)  .. (328.36bp,188bp) .. controls (323.55bp,179.73bp) and (323.21bp,169.03bp)  .. (u1*u0);
  \draw (337.36bp,178bp) node {$0$};
  \draw [black,->] (u2*u1*u0) ..controls (361.97bp,197.09bp) and (354.52bp,180.83bp)  .. (346.36bp,168bp) .. controls (344.18bp,164.57bp) and (341.6bp,161.05bp)  .. (u1*u0);
  \draw (366.36bp,178bp) node {$3$};
  \draw [black,->] (u1*u2*u0) ..controls (165.75bp,201.24bp) and (171.25bp,194.46bp)  .. (175.36bp,188bp) .. controls (181.16bp,178.88bp) and (186.6bp,168.03bp)  .. (u2*u0);
  \draw (198.86bp,178bp) node {$-1$};
  \draw [black,->] (u1*u2*u0) ..controls (150.66bp,196.79bp) and (146.32bp,179.85bp)  .. (153.36bp,168bp) .. controls (158.03bp,160.13bp) and (166.3bp,154.67bp)  .. (u2*u0);
  \draw (162.36bp,178bp) node {$2$};
  \draw [black,->] (u0*u1) ..controls (444.86bp,127bp) and (418.02bp,103.6bp)  .. (u1);
  \draw (443.36bp,110bp) node {$1$};
  \draw [black,->] (u0*u2*u0) ..controls (24.011bp,195.35bp) and (29.888bp,174.32bp)  .. (u0*u2);
  \draw (39.359bp,178bp) node {$0$};
  \draw [black,->] (u1*u2*u1*u0) ..controls (252.78bp,262.89bp) and (271.74bp,241.14bp)  .. (280.36bp,236bp) .. controls (298.67bp,225.09bp) and (322.05bp,219.02bp)  .. (u2*u1*u0);
  \draw (292.86bp,246bp) node {$-1$};
  \draw [black,->] (u1*u0) ..controls (318bp,132.89bp) and (312.29bp,126.06bp)  .. (307.36bp,120bp) .. controls (299.22bp,110bp) and (290.16bp,98.604bp)  .. (u0);
  \draw (316.36bp,110bp) node {$2$};
  \draw [black,->] (u0*u1*u2*u0) ..controls (179.19bp,274.39bp) and (374.92bp,261.94bp)  .. (387.36bp,256bp) .. controls (398.42bp,250.72bp) and (395.94bp,242.46bp)  .. (406.36bp,236bp) .. controls (421.53bp,226.6bp) and (440.65bp,220.64bp)  .. (u0*u1*u0);
  \draw (415.36bp,246bp) node {$1$};
  \draw [black,->] (u0*u1*u2*u0) ..controls (56.765bp,270.21bp) and (45.283bp,264.41bp)  .. (37.359bp,256bp) .. controls (30.034bp,248.23bp) and (25.515bp,237.07bp)  .. (u0*u2*u0);
  \draw (49.859bp,246bp) node {$-1$};
  \draw [black,->] (u0*u1*u2*u0) ..controls (42.039bp,272.73bp) and (10.711bp,265.75bp)  .. (3.3585bp,256bp) .. controls (-3.2495bp,247.24bp) and (1.5614bp,235.48bp)  .. (u0*u2*u0);
  \draw (12.359bp,246bp) node {$2$};
  \draw [black,->] (u0) ..controls (272.36bp,59.313bp) and (272.36bp,38.603bp)  .. (1);
  \draw (281.36bp,42bp) node {$1$};
  \draw [black,->] (u0*u2) ..controls (58.353bp,127bp) and (85.886bp,103.6bp)  .. (u2);
  \draw (96.359bp,110bp) node {$1$};
  \draw [black,->] (u0*u1*u2) ..controls (71.304bp,197.1bp) and (69.363bp,180.38bp)  .. (75.359bp,168bp) .. controls (77.5bp,163.58bp) and (80.883bp,159.69bp)  .. (u1*u2);
  \draw (84.359bp,178bp) node {$1$};
  \draw [black,->] (u1*u2*u0) ..controls (141.02bp,201.44bp) and (133.03bp,195.01bp)  .. (127.36bp,188bp) .. controls (120.42bp,179.43bp) and (114.67bp,168.43bp)  .. (u1*u2);
  \draw (136.36bp,178bp) node {$2$};
  \draw [black,->] (u2*u0) ..controls (216.1bp,127bp) and (241.92bp,103.6bp)  .. (u0);
  \draw (252.36bp,110bp) node {$0$};
  \draw [black,->] (u0*u1*u2*u0) ..controls (91.912bp,264.62bp) and (98.414bp,247.04bp)  .. (109.36bp,236bp) .. controls (114.86bp,230.45bp) and (121.91bp,225.94bp)  .. (u1*u2*u0);
  \draw (121.86bp,246bp) node {$-2$};
  \draw [black,->] (u0*u1*u2*u0) ..controls (115.08bp,269.95bp) and (126.36bp,264.07bp)  .. (134.36bp,256bp) .. controls (142.22bp,248.07bp) and (147.8bp,236.75bp)  .. (u1*u2*u0);
  \draw (156.36bp,246bp) node {$1$};
  \draw [black,->] (u0*u2*u1*u0) ..controls (506.45bp,270.58bp) and (522.15bp,264.63bp)  .. (534.36bp,256bp) .. controls (545.48bp,248.14bp) and (555.27bp,236.15bp)  .. (u0*u2*u1);
  \draw (563.36bp,246bp) node {$4$};
  \draw [black,->] (u0*u1*u0) ..controls (479.46bp,195.35bp) and (473.28bp,174.32bp)  .. (u0*u1);
  \draw (485.36bp,178bp) node {$2$};
  \draw [black,->] (u2*u1*u0) ..controls (380.83bp,195.04bp) and (396.98bp,173.07bp)  .. (u2*u1);
  \draw (407.36bp,178bp) node {$3$};
  \draw [black,->] (u1*u2*u1*u0) ..controls (238.36bp,263.45bp) and (238.36bp,242.73bp)  .. (u1*u2*u1);
  \draw (247.36bp,246bp) node {$3$};
  \draw [black,->] (u0*u2*u1*u0) ..controls (420.62bp,275.1bp) and (376.94bp,268.81bp)  .. (341.36bp,256bp) .. controls (324.14bp,249.8bp) and (322.91bp,241.21bp)  .. (305.36bp,236bp) .. controls (195.16bp,203.31bp) and (161.17bp,234.08bp)  .. (47.359bp,218bp) .. controls (47.257bp,217.99bp) and (47.155bp,217.97bp)  .. (u0*u2*u0);
  \draw (350.36bp,246bp) node {$4$};
  \draw [black,->] (u2*u1*u0) ..controls (336.51bp,201.85bp) and (317.49bp,195.2bp)  .. (301.36bp,188bp) .. controls (284.29bp,180.38bp) and (281.43bp,175.62bp)  .. (264.36bp,168bp) .. controls (249.94bp,161.56bp) and (233.21bp,155.57bp)  .. (u2*u0);
  \draw (310.36bp,178bp) node {$3$};
  \draw [black,->] (u0*u2*u1*u0) ..controls (458.79bp,264.46bp) and (444.28bp,246.77bp)  .. (428.36bp,236bp) .. controls (418.57bp,229.38bp) and (406.68bp,224.14bp)  .. (u2*u1*u0);
  \draw (459.36bp,246bp) node {$1$};
  \draw [black,->] (u0*u2*u1*u0) ..controls (420.37bp,274.59bp) and (378.99bp,267.91bp)  .. (369.36bp,256bp) .. controls (363.14bp,248.31bp) and (362.95bp,237.16bp)  .. (u2*u1*u0);
  \draw (378.36bp,246bp) node {$4$};
  \draw [black,->] (u1*u2*u1*u0) ..controls (217.86bp,263bp) and (189.64bp,239.6bp)  .. (u1*u2*u0);
  \draw (215.36bp,246bp) node {$3$};
  \draw [black,->] (u0*u2*u1*u0) ..controls (496.78bp,270.06bp) and (505.65bp,264.32bp)  .. (510.36bp,256bp) .. controls (514.74bp,248.26bp) and (513.92bp,244.14bp)  .. (510.36bp,236bp) .. controls (508.5bp,231.76bp) and (505.48bp,227.93bp)  .. (u0*u1*u0);
  \draw (521.36bp,246bp) node {$1$};
  \draw [black,->] (u0*u2*u1*u0) ..controls (473.77bp,263.45bp) and (478.03bp,242.73bp)  .. (u0*u1*u0);
  \draw (487.36bp,246bp) node {$4$};
\end{tikzpicture}
\end{center}
\caption{$\downgraph$ for $k=2$ truncated at the affine Grassmannian elements of
length $4$.}
\label{fig:downgraph}
\end{figure}

A \emph{strong strip} of length $i$ from $w$ to $v$, denoted by $w
\strongstrip v$, is a path 
\begin{align*}
            w
            \stackrel{\ell_1}\longrightarrow w_1
            \stackrel{\ell_2}\longrightarrow
            \cdots
            \stackrel{\ell_i}\longrightarrow w_i = v
\end{align*}
of length $i$ in $\downgraph$ with decreasing edge labels:
            $\ell_1 > \cdots > \ell_i$.
For non-negative integers $i$, define $D_i: \AA \to \AA$ as
\[
 D_i(\uu_w)
 = \sum_{
            w \strongstrip v
            \atop
            \size(w \strongstrip v) = i
        } \uu_{v}
\]
where the sum ranges over all strong strips of length $i$ that begin at
$w$. In particular, the coefficient of $\uu_{v}$ in $D_i(\uu_w)$ is the
number of strong strips of length $i$ that begin at $w$ and end at $v$.

\begin{Example}
With $k=2$, using the graph from Figure \ref{fig:downgraph}, one can verify that:
\begin{align*}
    D_1(\uu_0\uu_1\uu_2\uu_0) 
    &= 2\uu_0\uu_2\uu_0 + \uu_0\uu_1\uu_2 + 2\uu_1\uu_2\uu_0 + \uu_0\uu_1\uu_0;
    \\
    D_2(\uu_0\uu_1\uu_2\uu_0) 
    &= \uu_0\uu_2 + \uu_1\uu_2 + \uu_2\uu_0 + \uu_1\uu_0.
\end{align*}
\end{Example}

More generally, we define an operator $D_J$ for any composition $J$ of positive
integers; the operator $D_i$ defined above is $D_J$ for the composition $J=[i]$.
We need some additional notation. 
The \emph{ascent composition} of a sequence $\ell_1, \ell_2, \dots, \ell_m$ is
the composition
$[
i_1,
i_2 - i_{1},
\ldots,
i_j - i_{j-1},
m - i_j
]$,
where $i_1 < i_2 < \dots < i_j$ are the ascents of the sequence; that is, the elements in $\{1, \dots, m-1\}$ such
that $\ell_{i_a} < \ell_{i_a+1}$. For example, the ascent composition of the sequence $3,2,0,3,4,1$ is $[3,1,2]$ since the ascents are in positions $3$ and $4$.

If
$ w_0 \stackrel{\ell_1}\longrightarrow \cdots \stackrel{\ell_m}\longrightarrow w_m $
is a path in $\downgraph$, then we let 
$\ascomp(w_0 \stackrel{\ell_1}\longrightarrow \cdots \stackrel{\ell_m}\longrightarrow w_m)$
denote the ascent composition of the sequence of labels $\ell_1, \dots, \ell_m$.
It is a composition of the length of the path.

For a composition $J = [j_1, j_2, \dots, j_l]$ of positive integers, define
\[
 D_J(\uu_w)
 = \sum_{
         \ascomp\left(
            w 
            \stackrel{\ell_1}\longrightarrow w_1
            \stackrel{\ell_2}\longrightarrow
            \cdots
            \stackrel{\ell_m}\longrightarrow w_m
        \right)
        = J
        } \uu_{w_m}
\]
where the sum ranges over all paths in $\downgraph$ of length $m = j_1 + \dots
+ j_l$ beginning at $w$ whose sequence of labels has ascent composition $J$.

\begin{Example}
With $k=2$ one can verify using Figure \ref{fig:downgraph} that:
    \begin{align*}
    D_{[3]}(\uu_1\uu_2\uu_1\uu_0) &= \uu_2 + \uu_0, \\
    D_{[2,1]}(\uu_1\uu_2\uu_1\uu_0) & = \uu_2 + 2\uu_0 + \uu_1, \\
    D_{[1,2]}(\uu_1\uu_2\uu_1\uu_0) & = \uu_2 + 2\uu_0 + \uu_1, \\
    D_{[1,1,1]}(\uu_1\uu_2\uu_1\uu_0) & = \uu_0 + \uu_1.
    \end{align*}
\end{Example}

\section{ The details  }\label{sec:comb_desc}

A partition of the form $R = (c^{k+1-c})$, for some $c \in \{0, 1, \dots, k\}$,
is called a \emph{$k$-rectangle}. From now on, we let $\lambda$ be a partition
contained in a $k$-rectangle $R$. The goal of this section is to derive an
explicit combinatorial formula for $D_n(\uu_{w_\lambda})$. 

\subsection{ Goal of Section \ref{sec:comb_desc} }
Since $\lambda$ is contained inside a $k$-rectangle $R$, it is both a
$k$-bounded partition and a $(k+1)$-core, so we interchangeably think of
$\lambda$ as either. Under the bijection between $k$-bounded partitions and
$W^0$, the partitions contained in $R$ correspond to the elements of $W^0$ for
which any reduced expression does not contain an occurrence of the generator $s_c$.

If $(i,j)$ is a cell of $\lambda$, then its \emph{content} is $j-i$ and
its \emph{residue} is $(j-i) \mod (k+1)$. Since $\lambda$ is contained inside
$R=(c^{k+1-c})$, one can obtain $w_\lambda$ by reading the residues of the
cells in the diagram of $\lambda$ in the following order: start in the
bottom-right cell; read the residues in each row from right to left; and read
the rows from bottom to top. Explicitly, if $\lambda = (\lambda_1, \dots,
\lambda_l)$, then
\begin{align*}
    w_\lambda
    =
    \big(s_{\lambda_\ell-\ell} \, s_{\lambda_\ell-\ell-1} \cdots \, s_{1-\ell}\big)
    \cdots
    \big(s_{\lambda_i-i} \, s_{\lambda_i-i-1} \cdots \, s_{1-i}\big)
    \cdots 
    \big(s_{\lambda_1-1} \, s_{\lambda_1-2} \cdots \, s_{0}\big),
\end{align*}
where we write $s_a$, for any integer $a$, for the generator $s_{a \mod (k+1)}$
of $W$.

\begin{Definition}
    \label{d:wordfromdiagram}
For cells $x_1, \dots, x_n$ of a partition $\lambda$, let
$\lambda\setminus \{ x_1, \dots, x_n \}$ denote the diagram obtained by
removing $x_1, \dots, x_n$ from the diagram of $\lambda$.

For a diagram $\mathcal D$, let $w_{\mathcal D}$ denote the element of $W$ and
let $\uu_{\mathcal D}$ denote the element of $\AA$ obtained by reading the
residues of the cells in $\mathcal D$ in the order described above.
\end{Definition}

\begin{Example}\label{ex:331}
    Let $k = 5$, $\lambda = (3,3,1)$ and let $x$ be the cell $(2,2)$.
    Then $\lambda\setminus\{x\}$ is the diagram below (with each cell labelled
    by its residue).
    \begin{gather*}
    \tikztableausmall{{0,1,2},{5,\boldentry X,1},{4}}
    \begin{array}{ll}
        w_\lambda = s_4s_1s_0s_5s_2s_1s_0,
        & w_{\lambda \setminus \{x\}} = s_4s_1s_5s_2s_1s_0, \\
        \uu_\lambda = \uu_4\uu_1\uu_0\uu_5\uu_2\uu_1\uu_0,
        & \uu_{\lambda \setminus \{x\}} = \uu_4\uu_1\uu_5\uu_2\uu_1\uu_0.
    \end{array}
    \end{gather*}
\end{Example}

\begin{Remark}\label{remark:distinct}
    Since $\lambda$ is contained in $R=(c^{k+1-c})$, it follows immediately
    from the definition that the residues in each row and each column of
    $\lambda$ are distinct and do not include $c$. We will use this observation
    repeatedly in this section.
\end{Remark}

Recall the following theorem.

\begin{Theorem}\cite[Theorem 4.8]{BSS12}\label{thm:restriction}
    \label{sym-module-morphism}
    Suppose $w \in W$ and $v \in W_0$. Then $D_i(\uu_w \uu_v) = D_i(\uu_w) \uu_v$
    Consequently, $D_i$ is completely determined by its
    restriction to $\BB$.
\end{Theorem}

The goal of this section is to derive a combinatorial formula for
\begin{equation}
    \label{eq:Dn(wlambda)}
    D_n(\uu_\lambda) = \sum_{\size(w_\lambda \strongstrip w)=n} \uu_w
\end{equation}
that does not involve calculating strong strips.
By Theorem \ref{sym-module-morphism}, this would allow us to compute
$D_n(\uu_u)$ for any element $u \in W$ which factors as $w_\lambda v$ with
$\lambda\subseteq R$ and $v \in W_0$.

We begin by reformuating \eqref{eq:Dn(wlambda)}.
Recall that a strong strip $w_\lambda \strongstrip w$ of size $n$
is a path of length $n$ in $\downgraph$,
\begin{align*}
    w_\lambda
    \buildrel{\ell_1}\over\longrightarrow
    w_{1}
    \buildrel{\ell_2}\over\longrightarrow
    w_{2}
    \buildrel{\ell_3}\over\longrightarrow
    \cdots
    \buildrel{\ell_{n}}\over\longrightarrow
    w_{n} = w
\end{align*}
whose labels satisfy $\ell_1 > \ell_2 > \dots > \ell_n$.
Since each arrow comes from a strong cover, we obtain reduced expressions for
$w_1$, \dots, $w_n$ by starting with a reduced expression for $w_\lambda$ and
removing one letter at a time. It follows that there is a set
$\XX = \{x_1, \dots, x_n\}$ of cells of $\lambda$ for which
$w_a = w_{\lambda\setminus\{x_1,\dots,x_a\}}$ for $1 \leq a \leq n$.
Hence, we can write
\begin{align*}
    D_n(\uu_{\lambda})
    = \sum_{\XX} \uu_{\lambda\setminus\XX}
\end{align*}
where $\XX$ runs through some subset of the set of all collections of $n$ cells
from $\lambda$.

\begin{Theorem}
\label{thm:D_i}
If $\lambda$ is contained in $R=(c^{k+1-c})$, for some $c \in \{0,1,\dots,k\}$,
then
\begin{align*}
    D_n(\uu_{\lambda})
        = \sum_{\{x_1, \dots, x_n\}\in\mathfrak I}
          \uu_{{\lambda\setminus\{x_1, \dots, x_n\}}},
\end{align*}
where $\mathfrak I$ is the set of all collections of $n$ distinct cells
    $x_1 = (i_1,j_1)$, \dots, $x_n = (i_n,j_n)$
in $\lambda$ satisfying the following conditions:
\begin{description}
    \item[C1]
        if $a \neq b$, then $j_a \neq j_b$;
    \item[C2]
        if there exists $a$ and $b$ such that $i_a > i_b$ and $j_a < j_b$,
        then $(i_a, j_b) \notin \lambda$.
\end{description}
\end{Theorem}
The first condition (C1) says that no two cells among $x_1$, \dots, $x_n$
appear in the same column of $\lambda$; the second condition (C2) says that if
$x_a$ is to the southwest of $x_b$, then the rectangle delimited by $x_a$ and
$x_b$ is not contained in $\lambda$.

\begin{Example}
	Let $k = 4$ and suppose $\lambda = (2,2,1)$. Then $D_2(\uu_\lambda) = \uu_3\uu_4\uu_1 + \uu_0\uu_4\uu_0+\uu_4\uu_1\uu_0$ as shown in the picture below.
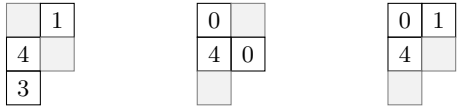
\[
\tikztableausmall{{X,1}, {4,X}, {3}} \hspace{.5in}
\tikztableausmall{{0,X}, {4,0}, {X}} \hspace{.5in}
\tikztableausmall{{0,1}, {4,X}, {X}}
\]
\end{Example}
The remainder of this section is devoted to proving this theorem.

\subsection{Labels of arrows and bounce paths}
Let $\XX = \{x_1, \dots, x_n\}$ be a collection of $n$ cells from $\lambda$,
and write $\XX_a = \{x_1, \dots, x_a\}$ for $a \leq n$.
We need to be able to compute the labels of arrows from
$w_{\lambda\setminus\XX_{a-1}}$ to $w_{\lambda\setminus\XX_{a}}$,
as well as find conditions on $\XX_a$ that guarantee the existence of
such an arrow.
Towards this end, consider the reduced expressions for
$w_{\lambda\setminus\XX_{a-1}}$
and
$w_{\lambda\setminus\XX_{a}}$
given by Definition \ref{d:wordfromdiagram}.
We can factor these words as
\begin{align}
    \label{eq:vudefn}
    & w_{\lambda\setminus\XX_{a-1}} = vs_{j_a-i_a}u,
    & w_{\lambda\setminus\XX_{a}} = vu,
\end{align}
where $s_{j_a-i_a}$ is the generator corresponding to the cell
$x_a = (i_a,j_a)$. Then
\begin{align}
    \label{eq:tij}
    w_{\lambda\setminus\XX_{a}}^{-1} w_{\lambda\setminus\XX_{a-1}} &= t_{u^{-1}(j_a-i_a), u^{-1}(j_a-i_a+1)}.
\end{align}
Hence, in order to have an arrow
we must have $u^{-1}(j_a-i_a) \leq 0 < u^{-1}(j_a-i_a+1)$.
The label of this arrow is
\begin{align*}
    w_{\lambda\setminus\XX_{a-1}}u^{-1}(j_a-i_a) &= vs_{j_a-i_a}(j_a-i_a) = v(j_a-i_a+1).
\end{align*}
We can compute $v(j_a-i_a+1)$ using a path in the diagram
$\lambda\setminus\XX_{a}$ that begins at $x_a$ and travels east and south,
turning whenever a cell in $\XX_{a}$ is encountered.

\begin{Definition}
    \label{def:ESpath}
    Let $\XX$ be a subset of the cells of $\lambda$.
    The \emph{East-South bounce path} of $x$ in the diagram
    $\lambda\setminus\XX$ is the path described by the
    following algorithm.
    Start at $x$ and
    travel East until you encounter a cell in $\XX$;
    then travel South until you encounter a cell in $\XX$;
    then travel East until you encounter a cell in $\XX$;
    and so on, until an East step puts you outside $\lambda$
    or a South steps hits the border cell of $\lambda$.
\end{Definition}

\begin{figure}[htb]
\begin{picture}(300,105)(0,50)
\put(90,100){
$\tikztableau[scale=0.35,every node/.style={font=\rm\scriptsize}]{
{,,,,,,,,,,},
{,$x$,,,,\boldentry X,,,},
{,,,,,,,},
{,\boldentry X,,,,\boldentry X,,},
{,,,,,,},
{$y$,,,\boldentry X,,,},
{,,,,,},
{,,,$-4$,},
}$
}
\put(117,130){\line(1,0){36}}
\put(153,130){\line(0,-1){20}}
\put(153,110){\vector(1,0){26}}
\put(180,107){\begin{small}$5$\end{small}}

\put(106,90){\line(1,0){28}}
\put(134,90){\vector(0,-1){17}}
\end{picture}
\caption{East-South bounce paths. In this case, the East-South bounce path of $x$ ends at a position of content $5$ and the East-South bounce path of $y$ ends at a position of content $-4$.}
\label{fig:sebouncepath}
\end{figure}
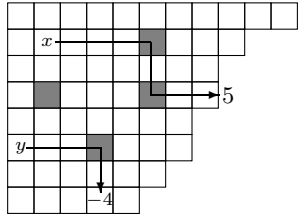

\begin{Lemma}
    \label{lemma:ESpath}
    Let $v$, $\XX_a$, $x_a = (i_a,j_a)$ be as above.
    Then $v(j_a-i_a+1)$ is the content of the last cell in the East-South
    bounce path of $x_a$ in $\lambda\setminus\XX_{a}$.
\end{Lemma}

\begin{proof}
    Write
    $v = \rho_{\ell(\lambda)} \rho_{\ell(\lambda)-1} \cdots \rho_{i_a+1} \rho_{i_a}$,
    where $\rho_i$ is the subword of
    \begin{align*}
       s_{-i+\lambda_i}
    \, s_{-i+\lambda_i-1}
    \cdots
    \, s_{-i+j+2}
    \, s_{-i+j+1}
    \, s_{-i+j}
    \cdots 
    \,
    \, s_{-i+2}
    \, s_{-i+1}
    \end{align*}
    corresponding to the $i$-th row of $\lambda\setminus\XX_a$.
    Since $\lambda$ is contained in a $k$-rectangle, there cannot be
    more than one occurrence of a generator $s_l$ in $\rho_i$
    (see Remark \ref{remark:distinct}).
    The result will follow by interpreting the following identity:
    \begin{align*}
        \rho_i(-i+j+1)
        =
        \begin{cases}
            -i+j+1, & \text{if neither $s_{-i+j}$ nor $s_{-i+j+1}$ occurs in $\rho_i$,} \\
            -i+j, & \text{if $s_{-i+j}$ occurs in $\rho_i$,} \\
            -i+j'+1,& \text{if $s_{-i+j+1}$ occurs in $\rho_i$,
                        but $s_{-i+j}$ does not,} \\
        \end{cases}
    \end{align*}
    where $s_{-i+j'} s_{-i+j'+1} \cdots s_{-i+j-1} s_{-i+j}$ occurs in $\rho_i$
    but $s_{-i+j'-1}$ does not.

    Assume the path is at the cell $(i,j)$ and that $s_{-i+j}$ does not occur
    in $\rho_i$ (as is the case at the beginning of the path). Suppose that the
    cell $(i,j+1)$ immediately to the right of $(i,j)$ is not contained in
    $\XX_a$. Two things happen: the path takes $d$ East steps, where $d$ is the
    number of cells that separate $(i,j)$ and the first cell in $\XX_a$ that
    lies to its right (or the number of cells in the row if no such cell in
    $\XX_a$ exists); and applying $\rho_i$ to $-i+j+1$ will increase it by $d$.
    Note that in the situation where no such cell in $\XX_a$ exists, the $d$
    East steps puts the bounce path just outside $\lambda$.

    Suppose instead that the cell $(i,j+1)$ is contained in $\XX_a$. Then
    $-i+j+1$ is fixed by $\rho_i$. If the cell $(i+1,j+1)$ immediately below
    $(i,j+1)$ is not contained in $\XX_a$, then two things happen: the path
    takes one South step; and applying $\rho_{i+1}$ to $-i+j+1$ decrements it
    by $1$. This will continue until the cell below is not contained in $\XX_a$
    or it is not contained in $\lambda$. In the former situation, we are back
    to the previous case; in the latter situation we reach the end of the
    East-South bounce path.

    Therefore, East and South steps in the path correspond to incrementing and
    decrementing $-i+j+1$ by $1$. But content also increases by $1$ with every
    East step and decreases by $1$ with every South step. Since the path starts
    at $x_a = (i_a,j_a)$, which has content $-i_a+j_a$, the content of the last
    cell in the path is $vs_{-i_a+j_a}(-i_a+j_a)$.
\end{proof}

This allows us to compute the label of an arrow of the form
$w_{\lambda\setminus\XX_{a-1}} \to w_{\lambda\setminus\XX_{a}}$.

\begin{Proposition}
    \label{prop:edgelabel}
    Suppose there is an arrow in $\downgraph$ from
    $w_{\lambda\setminus\XX_{a-1}}$ to $w_{\lambda\setminus\XX_{a}}$.
    Then its label is the content of the last cell in the East-South bounce
    path of $x_a$ in $\lambda\setminus\XX_{a}$.
\end{Proposition}

\begin{proof}
    By definition, the label is $v(-i_a+j_a+1)$, where $v$ is defined in
    \eqref{eq:vudefn}. The result follows from Lemma \ref{lemma:ESpath}.
\end{proof}

In a similar manner, $u^{-1}(j_a-i_a)$ and $u^{-1}(j_a-i_a+1)$ are the contents
of the last cells in West-North and North-West bounce paths, respectively,
starting at $\XX_a$.

\begin{Proposition}
    \label{prop:NWpaths}
    If $\XX = \{x_1, \dots, x_n\}$ is a collection of cells of $\lambda$, then
    $$w_{\lambda\setminus\XX_{a-1}} = w_{\lambda\setminus\XX_{a}} t_{\alpha,\beta},$$
    where $\alpha$ and $\beta$ are the contents of the last cells in the
    West-North and North-West bounce paths in $\lambda\setminus\XX_{a}$
    starting at $x_a$, respectively.
    In particular, if $\XX = \{(i,j)\}$,
    $$w_{\lambda} = w_{\lambda\setminus\{(i,j)\}} t_{-i+1,j}.$$
\end{Proposition}

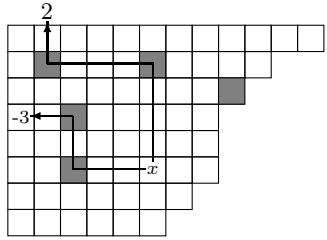
\begin{figure}[htb]
\begin{picture}(300,105)(0,50)
\put(90,100){
$\tikztableau[scale=0.35,every node/.style={font=\rm\scriptsize}]{
{,,,,,,,,,,,},
{,\boldentry X,,,,\boldentry X,,,,},
{,,,,,,,,\boldentry X},
{-3,,\boldentry X,,,,,},
{,,,,,,,},
{,,\boldentry X,,,$x$,,},
{,,,,,,},
{,,,,,},
}$
}
\put(155,93){\line(0,1){37}}
\put(155,130){\line(-1,0){40}}
\put(115,130){\vector(0,1){16}}
\put(112.5,147){\begin{small}$2$\end{small}}

\put(152,90){\line(-1,0){27.5}}
\put(124.5,90){\line(0,1){20}}
\put(124.5,110){\vector(-1,0){16}}
\end{picture}
\caption{North-West and West-North bounce paths. The North-West bounce path for $x$ ends in a cell of content $2$ and the West-North bounce path ends in a cell of content $-3$.}
\label{fig:nwbouncepath}
\end{figure}

\begin{Lemma}\label{hero_lemma}
    Let $\lambda$ be a Ferrers shape contained in a $k$-rectangle, and
    $\XX \subset \lambda$ a collection of cells satisfying (C2).
    Then $\ell(w_{\lambda\setminus\XX}) = \ell(w_{\lambda}) - |\XX|$;
    that is, the reading word $w_{\lambda\setminus\XX}$ of $\lambda\setminus\XX$
    is a reduced word in $W$.
\end{Lemma}

\begin{proof}
Let $|\lambda|$ denote the size of $\lambda$. We proceed by induction on $|\lambda|$. The result is clear for $|\lambda|=0, 1$. Now, assume that for some $n$, the result is true for all shapes $\lambda$ for which $|\lambda| \le n-1$, and all $\XX \subset \lambda$ satisfying (C2).

Let $\lambda$ be such that $|\lambda|=n$, and let $\XX \subset \lambda$ satisfy (C2). We aim to show that $w_{\lambda\setminus\XX}$ is a reduced word. Let $y$ be the rightmost box on the bottom row of $\lambda$ and let $\lambda'=\lambda \setminus \{y\}$. 
Let $i$ and $j$ be the row and column corresponding to $y$, so its content is $j-i$, and let $m = j-i (\mod k+1)$ be the residue of $y$.

We consider two cases, depending on when $y$ is in $\XX$ or not.

\textbf{Case 1:} $y \in \XX$.  Since, $|\lambda'| < |\lambda|$, by induction we can assume that $w_{\lambda' \setminus (\XX\setminus\{ y\})}$ is a reduced word. 
However, the cells in $\lambda \setminus \XX$  are precisely the same as the cells in $\lambda' \setminus (\XX \setminus \{y\})$.
Therefore $w_{\lambda \setminus \XX} = w_{\lambda' \setminus \XX'}$, and the result follows.

\textbf{Case 2:} $y \notin \XX$. Let $w = w_{\lambda' \setminus \XX}$; once again, by induction, we may assume that $w$ is a reduced word. We note that $(\lambda' \setminus \XX) \cup \{y\} = \lambda \setminus \XX$, so $w_{\lambda \setminus \XX} = s_m w$. It is well known that $s_m w$ is reduced if and only if $w^{-1}(m) < w^{-1}(m+1)$ (see for instance \cite[Proposition 8.3.6]{BB05}).

By Proposition \ref{prop:NWpaths}, $w_{\lambda \setminus \XX} = w t_{\alpha,\beta}$, where $\alpha$ and $\beta$ are the contents of the last cells in the West--North and North--West bounce paths in $\lambda \setminus \XX$ starting at $y$. Therefore, $s_m w = w t_{\alpha, \beta}$, so $w^{-1} s_m w = t_{\alpha,\beta}$. Thus, $\alpha = w^{-1}(m)$ and $\beta = w^{-1}(m+1)$. Now, since $\XX$ satisfies (C2), $\XX$ does not contain a box in the row containing $y$ and a box in the column containing $y$. Thus, between the West-North and the North-West paths in $\lambda \setminus \XX$ starting at $y$, at most one of them is not a straight line. Therefore, the last cell of the West-North path must lie on a diagonal at the left of the last cell of the North-West path, so $\alpha < \beta$. (For an example that illustrates this property, see Figure \ref{fig:bouncepaths}.) Thus, $w^{-1}(m) < w^{-1}(m+1)$, so we conclude that $w_{\lambda \setminus \XX}$ is reduced.
\end{proof}

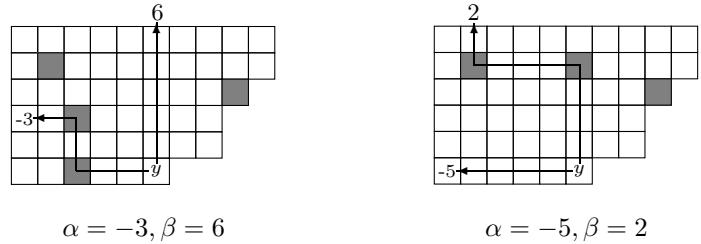
\begin{figure}[htb]
$
\begin{array}{cc}
\begin{picture}(150,65)(75,80)
\put(90,110){
$\tikztableau[scale=0.35,every node/.style={font=\rm\scriptsize}]{
{,,,,,,,,,},
{,\boldentry X,,,,,,,,},
{,,,,,,,,\boldentry X},
{-3,,\boldentry X,,,,,},
{,,,,,,,},
{,,\boldentry X,,,$y$},
}$
}
\put(155,93){\vector(0,1){53}}
\put(153,147){\begin{small}$6$\end{small}}

\put(152,90){\line(-1,0){27.5}}
\put(124.5,90){\line(0,1){20}}
\put(124.5,110){\vector(-1,0){16}}

\end{picture}
&
\begin{picture}(150,80)(75,80)
\put(90,110){
$\tikztableau[scale=0.35,every node/.style={font=\rm\scriptsize}]{
{,,,,,,,,,},
{,\boldentry X,,,,\boldentry X,,,,},
{,,,,,,,,\boldentry X},
{,,,,,,,},
{,,,,,,,},
{-5,,,,,$y$},
}$
}
\put(155,93){\line(0,1){37}}
\put(155,130){\line(-1,0){40}}
\put(115,130){\vector(0,1){16}}
\put(112.5,147){\begin{small}$2$\end{small}}

\put(152,90){\vector(-1,0){43.5}}
\end{picture}\\
\alpha = -3, \beta=6 &
\alpha = -5, \beta=2 \\
\end{array}$
\caption{Examples of the two possibilities for the North-West and the West-North bounce paths starting at $y$ and the corresponding $\alpha, \beta$.}
\label{fig:bouncepaths}
\end{figure}

\subsection{Grassmannian factorizations}
Here we describe the factorization of $w_{\lambda\setminus\{x\}}$ as a product
$w^{(0)} w_{(0)}$ with $w^{(0)} \in W^0$ and $w_{(0)} \in W_0$. 

For a cell $x = (i,j) \in \lambda$, let $\mathcal H_x$ denote cells in the hook
of $x$; that is, the cells directly below $x$, the cells directly to the right
of $x$, and the cell $x$ itself. If $\lambda'$ denotes the conjugate of
$\lambda$, then $\mathcal H_x$ is
\begin{gather}
    \label{eq:hookcells}
    \{
    (\lambda'_j,j), (\lambda'_j-1,j), \ldots, (i+1,j),
    (i,\lambda_i), (i,\lambda_i-1), \ldots, (i,j+1),
    (i,j)
    \}.
\end{gather}
As illustrated in Figure \ref{fig:removehook}, the cells in
$\lambda\setminus\mathcal H_x$ that are below $\mathcal H_x$ can be shifted up one
cell and to the left one cell, resulting in a partition $\lambda_{x}$.
Moreover, this shift does not change the residues of these cells so that
$w_{\lambda_x} = w_{\lambda\setminus\mathcal H_x}$.

Let $\Gamma_{\lambda,x}$ denote the reduced expression obtained by reading from
right to left the residues in the first row that are above this hook, and the
residues in the first column from bottom to top that are to the left of the
hook, but not including those in the row or column of $x$. Explicitly: 
\begin{gather*}
    \Gamma_{\lambda,x} = 
    \big(s_{\lambda_i-1} s_{\lambda_i-2} \cdots s_{j+1} s_j\big)
    \big(s_{-\lambda'_j+1} s_{-\lambda'_j+2} \cdots s_{-i}\big)
\end{gather*}

\begin{figure}[htb]
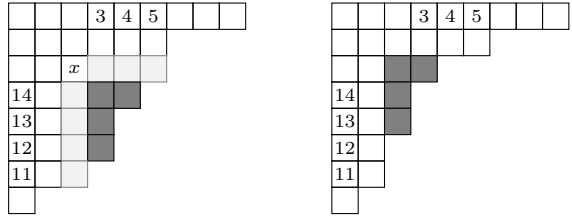

\begin{gather*}
\tikztableau[scale=0.35,every node/.style={font=\rm\scriptsize}]{
{,,,3,4,5,,,},
{,,,,,},
{,,$x$,X,X,X},
{14,,X,\boldentry X,\boldentry X},
{13,,X,\boldentry X},
{12,,X,\boldentry X},
{11,,X},
{\,},
}
\qquad
\tikztableau[scale=0.35,every node/.style={font=\rm\scriptsize}]{
{,,,3,4,5,,,},
{,,,,,},
{,,\boldentry X,\boldentry X},
{14,,\boldentry X},
{13,,\boldentry X},
{12,},
{11,},
{\,},
}
\end{gather*}
\caption{For $k=16$, $\lambda = (9,6,6,5,4,4,3,1)$ and $x = (3,3)$,
we have $\lambda_x = (9,6,4,3,3,2,2,1)$,
$\Gamma_{\lambda,x} = s_{5}s_{4}s_{3} \cdot s_{11}s_{12}s_{13}s_{14}$.}
\label{fig:removehook}
\end{figure}

\begin{Lemma}
    \label{lemma:split-1cell}
    If $x=(i,j)$ is a cell of $\lambda$, then
    \[w_{\lambda \setminus \{x\}} = w_{\lambda_x} \Gamma_{\lambda,x}.\]
\end{Lemma}

\begin{proof}
    Let $\mathcal H_x$ be the set in \eqref{eq:hookcells}, with the ordering given as written.
    Repeated application of Proposition \ref{prop:NWpaths} gives
    \begin{align*}
        w_\lambda =
        w_{\lambda\setminus\mathcal H_x} t_{-i+1,j}
        \Big(t_{-i+1,j+1} \cdots t_{-i+1,\lambda_i-1} \, t_{-i+1,\lambda_i} \Big)
        \Big(t_{-i,j} \cdots t_{-\lambda'_j+2,j} \, t_{-\lambda'_j+1,j} \Big).
    \end{align*}
    Repeated application of the identity $yt_{\alpha,\beta} =
    t_{y(\alpha),y(\beta)}y$ with $y = t_{-i+1,j}$ yields
    \begin{align*}
        w_\lambda
        &= w_{\lambda\setminus\mathcal H_x}
        \Big(t_{j,j+1} \cdots t_{j,\lambda_i-1} \, t_{j,\lambda_i} \Big)
        \Big(t_{-i,-i+1} \cdots t_{-\lambda'_j+2,-i+1} \, t_{-\lambda'_j+1,-i+1} \Big)
        t_{-i+1,j}
        \\
        &= w_{\lambda\setminus\mathcal H_x}
        \Big(s_{\lambda_i-1} s_{\lambda_i-2} \cdots s_{j+1} s_j\Big)
        \Big(s_{-\lambda'_j+1} s_{-\lambda'_j+2} \cdots s_{-i-1} s_{-i}\Big)
        t_{-i+1,j}
        \\
        &= w_{\lambda\setminus\mathcal H_x}
        \, \Gamma_{\lambda,x} \,
        t_{-i+1,j},
    \end{align*}
    where the second equality follows from the fact that \[t_{a,b} = s_a s_{a+1} \cdots s_{b-2} s_{b-1} s_{b-2} \cdots s_{a+1} s_a = s_{b-1} s_{b-2} \cdots s_{a+1} s_{a} s_{a+1} \cdots s_{b-2} s_{b-1}.\]
    By Proposition \ref{prop:NWpaths}, we have
        $w_{\lambda\setminus\{x\}} = w_{\lambda\setminus\mathcal H_x} \Gamma_{\lambda,x}$.
    So it remains to show that
        $w_{\lambda\setminus\mathcal H_x} = w_{\lambda_{x}}$,
    but this was observed above.
\end{proof}

\subsection{From strong strips to subsets of cells}

As mentioned above, each strong strip starting at $\lambda$ corresponds to removing some cells of the diagram of $\lambda$. The goal of this sub-section is to show that the cells $\XX$ corresponding to a given strong strip must satisfy (C1) and (C2) of Theorem \ref{thm:D_i}.

We start with a simple lemma about cycles in the symmetric group, which will be used below.

\begin{Lemma}\label{le:reflections}
Let $a<b$ and let $c_{a,b}$ be the cycle $(a,a+1,\ldots,b) = s_a s_{a+1} \cdots s_{b-1}$. Then
$c_{a,b}^2$ is not reduced.
\end{Lemma}
\begin{proof}

The proof is by induction on $b-a$. Note that if $b-a = 1$, then $c_{a,b} = s_a$, and thus, $c_{a,b}^2$ is clearly not reduced. Now, for the induction step, assume that for $b-a = n$, the word is not reduced, and consider $a,b$ for which $b-a = n+1$. Thus,
\begin{eqnarray*}
c_{a,b}^2	& = & s_a s_{a+1} ááás_{b-1} s_a s_{a+1}\cdots s_{b-1} \\
		& = & s_a s_{a+1}s_a \cdots s_{b-1}s_{a+1}\cdots s_{b-1} \text{ since $s_a$ commutes with $s_i$ for $i \ge a+2$} \\
		& = & (s_{a+1} s_a s_{a+1}) \cdots s_{b-1}s_{a+1}\cdots s_{b-1} \\
		& = & s_{a+1} s_a c_{a+1,b}^2
\end{eqnarray*}
which by induction, is not reduced.
\end{proof}

Now, we let
$$
    w_\lambda
    \buildrel{\ell_1}\over\longrightarrow
    w_{\lambda \setminus \XX_1}
    \buildrel{\ell_2}\over\longrightarrow
    w_{\lambda \setminus \XX_2}
    \buildrel{\ell_3}\over\longrightarrow
    \cdots
    \buildrel{\ell_{n}}\over\longrightarrow
    w_{\lambda \setminus \XX_n}
$$
be a path with $\ell_1 > \cdots > \ell_n$.

We first prove that $\XX$ satisfies a weaker version of (C2), which we use to
prove that $\XX$ satisfies (C1). This weaker version of (C2) combined with (C1)
implies (C2).

\begin{Lemma}
    \label{lemma:rectangle}
    Let $\XX = \{x_1, \dots, x_n\} \subseteq \lambda$.
    Suppose there exists two cells $x_a = (i_a,j_a)$ and $x_b=(i_b,j_b)$ in $\XX$
    with $i_a > i_b$, $j_a < j_b$, and $(i_a, j_b) \in \lambda$.
    If no other cell in $\XX$ is contained in the rectangle delimited by $x_a$
    and $x_b$, then 
    $\ell(w_{\lambda\setminus\XX}) < \ell(w_\lambda) - n$.
    In this case, there cannot be a strong $n$-strip from $w_\lambda$ to
    $w_{\lambda\setminus\XX}$ in $\downgraph$.
\end{Lemma}

\begin{proof}
Let $\mathcal R$ be the rectangle delimited by $x_a$ and $x_b$. It is enough to
show that the expression for $w_{\mathcal R \setminus \{x_a, x_b\}}$
given by Definition \ref{d:wordfromdiagram} is not reduced.

As illustrated in the diagram below,
let $\mathcal T$ denote the cells in the first row of $\mathcal R$ not
including $x_b$, and let $w_{\mathcal T}$ be the correspond word.
Let $\mathcal B$ denote the cells in the last row of $\mathcal R$ not
including $x_a$, and let $w_{\mathcal B}$ be the correspond word.
Let $\mathcal M$ denote all but the first row and last row of $\mathcal R$.
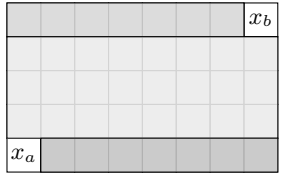
\begin{align*}
    \begin{tikzpicture}[scale=0.45,every node/.style={font=\rm\small}]
    \tikztableauinternal{
    {,,,,,,,$x_b$},
    {,,,,,,,},
    {,,,,,,,},
    {,,,,,,,},
    {$x_a$,,,,,,,},
    }
    \draw [fill=lightgray!60,opacity=.90] (0,-1) rectangle +(7,1);
    \draw [fill=lightgray!30,opacity=.90] (0,-4) rectangle +(8,3);
    \draw [fill=lightgray!90,opacity=.90] (1,-5) rectangle +(7,1);
    \end{tikzpicture}
\end{align*}
Lemma \ref{lemma:split-1cell} applied to the cell $x_a$ inside the diagram
formed by $\mathcal M$ and $\mathcal B$ shows that
$w_{\mathcal B} w_{\mathcal M}$ factors as $w_{\mathcal M} \Gamma'$, where
$\Gamma'$ is the reading word of the diagram consisting of the cells in the
first row of $\mathcal M$ except for the first cell.
This diagram is a diagonal shift of $\mathcal T$, and so
$\Gamma' = w_{\mathcal T}$.
This shows that
$w_{\mathcal B} w_{\mathcal M} = w_{\mathcal M} w_{\mathcal T}$.
Then $w_{\mathcal R \setminus \{x_a, x_b \}}
= w_{\mathcal B} w_{\mathcal M} w_{\mathcal T}
= w_{\mathcal M} (w_{\mathcal T} w_{\mathcal T})$.
Note that $w_{\mathcal T} = s_i s_{i+1}\cdots s_{i+m}$ for some $i$ and some $m \leq k$.
Hence, this word is not reduced, since for any $i$ and any $m\leq k$, the word
$(s_i s_{i+1}\cdots s_{i+m})^2$ is not reduced by Lemma \ref{le:reflections}.

The last statement follows from the fact that a strong $n$-strip should decrease the length by exactly $n$.
\end{proof}

\begin{Lemma}
    \label{lemma:C1holds}
    If
    $
        w_\lambda
        \buildrel{\ell_1}\over\longrightarrow
        w_{\lambda \setminus \XX_1}
        \buildrel{\ell_2}\over\longrightarrow
        \cdots
        \buildrel{\ell_{n}}\over\longrightarrow
        w_{\lambda \setminus \XX_n}
    $
    is a path in $\downgraph$
    with $\ell_1 > \ell_2 > \dots > \ell_n$,
    then no two cells in $\XX_n$ are in the same column;
    that is, (C1) holds.
\end{Lemma}

\begin{proof}
Pick $a$ so that $x_1, \dots, x_{a-1}$ are in different columns
and $x_a$ is in the same column as $x_b$ for some $b \in \{1, \dots, a-1\}$.

Suppose $x_a$ is above $x_b$. If there is no cell of $\XX_{a-1}$ directly
to the right of $x_a$, then by Proposition \ref{prop:edgelabel} the label $\ell_a$
is greater than $\ell_b$, a contradiction. So suppose instead that there is a
cell $x_d \in \XX_{a-1}$ directly to the right of $x_a$. If there is more than
one, let $x_d$ be the leftmost one. By assumption, no cell of $\XX_{a-1}$ lies below $x_d$, so Proposition \ref{prop:edgelabel} implies that the label
$\ell_a$ is the content of the last cell in the column containing $x_d$, and
$\ell_b$ is bounded above by $1$ plus the content of the last cell in the row
containing $x_b$. Since $\ell_a < \ell_b$, it follows that the rectangle with
vertices $x_b$ and $x_d$ is contained in $\lambda$, contradicting Lemma
\ref{lemma:rectangle}. Therefore, $x_a$ is not above $x_b$.

Suppose $x_a$ is below $x_b$. We argue first that there is no cell of
$\XX_{a-1}$ to the left of $x_a$ or $x_b$. If there were a cell $x_d \in
\XX_{a-1}$ to the left of $x_a$, then the rectangle with vertices $x_b$ and
$x_d$ would be contained in $\lambda$, contradicting Lemma
\ref{lemma:rectangle}.

Suppose there is a cell $x_d \in \XX_{a-1}$ to the left of $x_b$. Of all such
cells, pick $x_d$ so that there are no cells of $\XX_{a-1}$ between $x_d$ and
$x_b$. 

Assume first that $d > b$. Then $\ell_d$ is the content of the last
cell in the column containing $x_a$ and $x_b$. But then $\ell_d < \ell_a$,
contradicting that $\ell_a < \ell_d$.

Suppose instead that $d < b$. We have two subcases to consider.

\emph{Case 1:} there is no cell of $\XX_d$ to the right of $x_d$.
Let $x_f$ be the rightmost cell of $\XX_a$ that lies in the row containing
$x_d$. Note that $x_f \neq x_d$ since $x_b$ lies to the right of $x_d$.
The East-South bounce paths beginning at $x_d$ and $x_f$ travel East and they
do not turn. In particular, they are coterminus, implying that $\ell_f
= \ell_d$, contradicting $\ell_f \neq \ell_d$.

\emph{Case 2:} there is a cell of $\XX_d$ to the right of $x_d$.
The East-South bounce path starting at $x_d$ bounces South at some
cell $x_e \in \XX_d$. By assumption, there are no cells in $\XX_a$
below $x_e$, so $\ell_d$ is the content of the last cell of the column
containing $x_e$.
Let $f$ be the first index for which $x_f$ lies between $x_d$ and $x_e$ (in the
same row); such a cell exists since $x_b$ lies between $x_d$ and $x_e$.
The East-South bounce path starting at $x_f$ also bounces South at $x_e$ and,
as above, $\ell_f$ is the content of the last cell of the column containing
$x_e$. Thus, $\ell_f = \ell_d$, contradicting $\ell_f \neq \ell_d$.

Therefore, there is no cell of $\XX_{a-1}$ to the left of $x_a$ or $x_b$.
Since no cell of $\XX_{a-1}$ is to the left of $x_a$ or $x_b$, we can use
Proposition \ref{prop:NWpaths} to compute a pair $(\alpha,\beta)$ satisfying
$t_{\alpha,\beta} = w_{\lambda\setminus\XX_{a}}^{-1} w_{\lambda\setminus\XX_{a-1}}$, where $\alpha < \beta$,
$\alpha \in \{-1,-2,\dots,-k+c\}$,
and $\beta \in \{0,-1,-2,\dots,-k+c+1\}$.
In particular, there is no such pair $(\alpha,\beta)$ with $\alpha \leq 0 <
\beta$. This contradicts the fact that there is an arrow in $\downgraph$ from
$w_{\lambda\setminus\XX_{a-1}}$ to $w_{\lambda\setminus\XX_{a}}$.
Therefore, $x_a$ is not below $x_b$.
\end{proof}

\begin{Proposition}
    \label{prop:C1andC2hold}
    If
    $
        w_\lambda
        \buildrel{\ell_1}\over\longrightarrow
        w_{\lambda \setminus \XX_1}
        \buildrel{\ell_2}\over\longrightarrow
        w_{\lambda \setminus \XX_2}
        \buildrel{\ell_3}\over\longrightarrow
        \cdots
        \buildrel{\ell_{n}}\over\longrightarrow
        w_{\lambda \setminus \XX_n}
    $
    is a path in $\downgraph$
    with $\ell_1 > \ell_2 > \dots > \ell_n$,
    then $\XX_n$ satisfies (C1) and (C2).
\end{Proposition}

\begin{proof}
    $\XX_n$ satisfies (C1) by Lemma \ref{lemma:C1holds}.
    Suppose $\XX_n$ contains a pair of cells $x_a$ and $x_b$ violating (C2).
    Since $\XX_n$ satisfies (C1), of all such pairs $x_a,x_b$, we choose one so
    that so that the smallest rectangle $\mathcal R$ containing $x_a$ and $x_b$
    does not contain any other cells from $\XX_n$. The result then follows from
    Lemma \ref{lemma:rectangle}.
\end{proof}

\subsection{From subsets of cells to strong strips}
This sub-section is devoted to proving that each collection of cells $\XX$ satisfying (C1) and (C2) defines a strong
strip from $w_\lambda$ to $w_{\lambda\setminus\XX}$.

\begin{Proposition}\label{prop:ordering}
    Suppose $\XX \subseteq \lambda$ satisfies (C1) and (C2).
    Then there exists a unique ordering of the cells $x_1, \dots, x_n$ of $\XX$
    so that there is a path in $\downgraph$
    \begin{align}
        \label{eq:npath}
        w_\lambda
        \buildrel{\ell_1}\over\longrightarrow
        w_{\lambda \setminus \{x_1\}}
        \buildrel{\ell_2}\over\longrightarrow
        \cdots
        \buildrel{\ell_{n}}\over\longrightarrow
        w_{\lambda \setminus \{x_1, \dots, x_{n}\}}
    \end{align}
    with $\ell_1 > \ell_2 > \cdots > \ell_{n}$.
\end{Proposition}


\begin{proof} Suppose we have an ordering of the cells in $\XX = \{ x_1, x_2, \dots, x_n\}$.
    Let $\mathfrak c_a$ denote the content of the last cell in the East-South bounce path in
    $\lambda\setminus\XX_a$ which starts at $x_a$. We claim that the ordering of the $x_i$ can be chosen so that $\mathfrak c_1 > \cdots >
    \mathfrak c_n$. By Proposition \ref{prop:edgelabel}, these contents will be
    the labels in \eqref{eq:npath}.

    Since no two cells of $\XX$ are in the same column, the East-South bounce
    paths depend only on the order of the cells of $\XX$ in each row.
    Let $x_{a_1}$, $x_{a_2}$, \dots, $x_{a_{t-1}}$, $x_{a_t}$, be the cells of
    $\XX$ in the $i$-th row of $\lambda$ listed from right to left.
   The contents of the last cells in the East-South bounce paths starting at these cells
    are, respectively,
    \begin{gather*}
        \lambda_{i_{a_1}} - i_{a_1} + 1 >
        j_{a_1} - \lambda'_{j_{a_1}} >
        \dots >
        j_{a_{t-1}} - \lambda'_{j_{a_{t-1}}}.
    \end{gather*}
    Note that these are all distinct. Moreover, none of the above contents are equal to the
    content coming from another East-South bounce path. To see this, first note that the cells in $\lambda$ which reside at the end of an East-South bounce path all have distinct contents. Then it is enough to assume that another
    bounce path ends at $(\lambda'_{j_{a_r}},j_{a_r})$, in which case it must originate
    in a preceding row and turn South in the $j_{a_r}$-th column. This could
    only happen if there is another cell in this column, contradicting (C1).
    Order the cells in $\XX$ according to these values, in decreasing order.
    Note that any other ordering of these cells will result in a sequence of
    contents $\mathfrak c_1, \dots, \mathfrak c_n$ that is not decreasing.
    This proves the uniqueness of the path.

    It remains to show that this ordering defines a path in $\downgraph$; that
    is, we need to show that there is an arrow from
    $w_{\lambda\setminus\XX_{a-1}}$ to $w_{\lambda\setminus\XX_{a}}$.
    With respect to the above order, no cell of $\XX_{a-1}$ is directly to the left of
    or directly above $x_a$. Hence,
    $w_{\lambda\setminus\XX_{a}} = w_{\lambda\setminus\XX_{a-1}} t_{-i_a+1,j_a}$
    by Proposition \ref{prop:NWpaths},
    and so $-i_a+1 \leq 0 < j_a$.
    %
    %
    It remains to show that
    $\ell(w_{\lambda\setminus\XX_{a}}) = \ell(w_{\lambda\setminus\XX_{a-1}})-1$. However, since each $\XX_a$ satisfies (C1) and (C2),  Lemma \ref{hero_lemma} implies that $\ell(w_{\lambda \setminus \XX_a}) = \ell(w_\lambda) - a$ for all $a$, which implies that $\ell(w_{\lambda\setminus\XX_{a}}) = \ell(w_{\lambda\setminus\XX_{a-1}})-1$.
    \end{proof}

By Proposition \ref{prop:ordering}, every collection of cells $\XX$ satisfying (C1) and (C2) defines a strong strip and by Proposition \ref{prop:C1andC2hold}, each strong strip corresponds to such an $\XX$. This proves Theorem \ref{thm:D_i}. 

\subsection{Combinatorial formula for $D_{1^n}$}

Adapting the proof of Theorem \ref{thm:D_i} by working with columns instead of
rows and decreasing labels instead by increasing labels, we obtain a
combinatorial formula for the expansion of $D_{1^n}(\uu_\lambda)$.

\begin{Theorem}
\label{thm:D_1^n}
If $\lambda$ is contained in $R=(c^{k+1-c})$, for some $c \in \{0,1,\dots,k\}$,
then
\begin{align*}
    D_{1^n}(\uu_{\lambda})
        = \sum_{\{x_1, \dots, x_n\}}
          \uu_{{\lambda\setminus\{x_1, \dots, x_n\}}},
\end{align*}
where the sum ranges over all sets of $n$ cells from $\lambda$ that satisfy
(C2) and that do not contain two cells in the same row.
\end{Theorem}

\section{Expansions of non-commutative $k$-Schur functions}\label{sec:expansions}

We end this paper by using our computations of $D$ operators to give expansions
of $k$-Schur functions, generalizing Theorem 4.6 of \cite{BSS11}, which in turn
was an extension of Berg, Bergeron, Thomas, and Zabrocki's \cite{BBTZ11}
expansion for the rectangle $R = (c^{k+1-c})$. The main result of this section
is an expansion of $\nckschur_\lambda$ for $\lambda = (c^{k-c},c-i)$ and
$i\geq0$. The case $i=0$ was first established in \cite{BBTZ11}, the case $i=1$
was established in \cite{BSS11}. We begin with a remark.

\begin{Remark}
\label{remark:otherresidues}
From the definition of the marked strong order graph $\downgraph$, there is an
edge from $x$ to $y$ for every pair $(i,j)$ satisfying $y \, t_{i,j} = x$ and
$i \leq 0 < j$. As was pointed out in \cite{LLMS10}, one can fix an integer $m
\in \{0, 1, \dots, k\}$ and work instead with transpositions $t_{i,j}$
satisfying $i\leq m<j$. In an analogous way one obtains operators $D_n^{(m)}$
on $\AA$, and the results and arguments in this article hold for these
operators (one needs only work with $m$-Grassmannian elements instead of
$0$-Grassmannian elements). The operator $D_n^{(m)}$ for $m \neq 0$
is \emph{not} equal to $D_n$, however, by Theorem~\ref{thm:restriction}, their
restrictions to the affine Fomin--Stanley subalgebra $\BB$ coincide:
explicitly, $D_n^{(m)}(\bb) = D_n(\bb)$ for all $m$ and all $\bb \in \BB$.
\end{Remark}

We next recall the expansion described in \cite{BBTZ11}.

\begin{Theorem}[\cite{BBTZ11}, Theorem 4.12]\label{thm:rectangle}
Suppose $R = (c^{k+1-c})$.
The non-commutative $k$-Schur function $\nckschur_R$ has the expansion: \[\nckschur_R = \sum_{\lambda \subseteq R} \uu_{{R\cup\lambda / \lambda}}.\]
\end{Theorem}

\begin{Example}\label{eg:33}
Let $R = (3,3)$ and $k=4$. Then $\nckschur[4]_R$ is the sum of all the monomials in
$\uu_i$ corresponding to the reading words of the skew-partitions $(R\cup
\lambda) / \lambda$, where $\lambda$ is a partition contained inside the
rectangle $R$, as shown:

\begin{gather*}
\begin{array}{ccccc}
  \tikztableausmall{{0,1,2},{4,0,1},}
& \tikztableausmall{{X,1,2},{4,0,1},{\boldentry 3},} 
& \tikztableausmall{{X,1,2},{X,0,1},{\boldentry 3},{\boldentry 2},}
& \tikztableausmall{{X,X,2},{4,0,1},{\boldentry 3,\boldentry 4},}
& \tikztableausmall{{X,X,2},{X,0,1},{\boldentry 3,\boldentry 4},{\boldentry 2},}
\\
  \uu_1\uu_0\uu_4\uu_2\uu_1\uu_0
& \uu_3\uu_1\uu_0\uu_4\uu_2\uu_1
& \uu_2\uu_3\uu_1\uu_0\uu_2\uu_1
& \uu_4\uu_3\uu_1\uu_0\uu_4\uu_2
& \uu_2\uu_4\uu_3\uu_1\uu_0\uu_2
\\[1em]
  \tikztableausmall{{X,X,X},{4,0,1},{\boldentry 3,\boldentry 4,\boldentry 0},}
& \tikztableausmall{{X,X,2},{X,X,1},{\boldentry 3,\boldentry 4},{\boldentry 2,\boldentry 3},}
& \tikztableausmall{{X,X,X},{X,0,1},{\boldentry 3,\boldentry 4,\boldentry 0},{\boldentry 2},}
& \tikztableausmall{{X,X,X},{X,X,1},{\boldentry 3,\boldentry 4,\boldentry 0},{\boldentry 2,\boldentry 3},}
& \tikztableausmall{{X,X,X},{X,X,X},{\boldentry 3,\boldentry 4,\boldentry 0},{\boldentry 2,\boldentry 3,\boldentry 4},}
\\
  \uu_0\uu_4\uu_3\uu_1\uu_0\uu_4
& \uu_3\uu_2\uu_4\uu_3\uu_1\uu_2
& \uu_2\uu_0\uu_4\uu_3\uu_1\uu_0
& \uu_3\uu_2\uu_0\uu_4\uu_3\uu_1
& \uu_4\uu_3\uu_2\uu_0\uu_4\uu_3
\end{array}
\end{gather*}

\end{Example}

We are now ready to state our main result on expansions of non-commutative $k$-Schur functions.

\begin{Theorem}
\label{thm:rectangle-strip}
Let $0 \leq i \leq c$.
\[
    \nckschur_{(c^{k-c}, c-i)} =
    \sum_{\lambda \subseteq (c^{k+1-c})}
    \sum_{ \{ x_1, x_2, \dots, x_i \}}
        \uu_{(R \cup \lambda / \lambda) \setminus \{x_1, x_2, \dots, x_i\}},
\]
where the second sum is over all collections of $i$ cells in the rows below $R$ satisfying (C1) and (C2).
\end{Theorem}

\begin{proof}
  Theorem 4.7 of \cite{BSS12} shows that $D_i( \nckschur_{(c^{k-c+1})} ) =  \nckschur_{(c^{k-c}, c-i)}  $.  By Remark \ref{remark:otherresidues}, $D_i = D_i^{(c)}$ on $\BB$. The statement now follows from applying Theorem \ref{thm:D_i} with $D_i^{(c)}$ to Theorem \ref{thm:rectangle}.
\end{proof}

\begin{Example}\label{eg:31}
    Let $R = (3,3)$ and $k=4$. Then $\nckschur[4]_{(3,1)}$ is the sum of all the monomials which occur after applying $D_2^{(3)}$ to the shapes starting in row $3$ in Example \ref{eg:33} (i.e. the red part). The first three terms in Example \ref{eg:33} are zero after application of $D_2^{(3)}$. The fourth term yields one monomial:
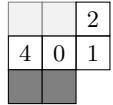
\begin{gather*}
\begin{array}{ccc}
 \tikztableausmall{{X,X,2},{4,0,1},{\boldentry X,\boldentry X},}
\\
 \uu_1\uu_0\uu_4\uu_2
\end{array}
\end{gather*}

The fifth term yields two:
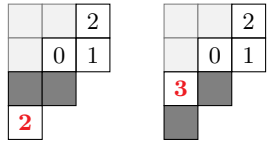
\begin{gather*}
\begin{array}{cc}
 \tikztableausmall{{X,X,2},{X,0,1},{\boldentry X,\boldentry X},{\boldentry 2},}
& \tikztableausmall{{X,X,2},{X,0,1},{\boldentry 3,\boldentry X},{\boldentry X},}
\\
 \uu_2\uu_1\uu_0\uu_2
& \uu_3\uu_1\uu_0\uu_2
\end{array}
\end{gather*}

The sixth term yields three:
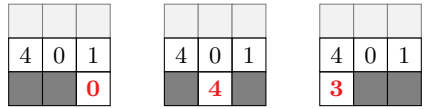
\begin{gather*}
\begin{array}{ccc}
  \tikztableausmall{{X,X,X},{4,0,1},{\boldentry X,\boldentry X,\boldentry 0},}
&   \tikztableausmall{{X,X,X},{4,0,1},{\boldentry X,\boldentry 4,\boldentry X},}
&   \tikztableausmall{{X,X,X},{4,0,1},{\boldentry 3,\boldentry X,\boldentry X},}
\\
 \uu_0\uu_1\uu_0\uu_4
& \uu_4\uu_1\uu_0\uu_4
& \uu_3\uu_1\uu_0\uu_4
\end{array}
\end{gather*}

The seventh term yields three:

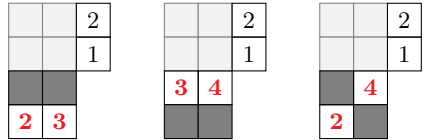
\begin{gather*}
\begin{array}{ccc}
 \tikztableausmall{{X,X,2},{X,X,1},{\boldentry X,\boldentry X},{\boldentry 2,\boldentry 3},}
& \tikztableausmall{{X,X,2},{X,X,1},{\boldentry 3,\boldentry 4},{\boldentry X,\boldentry X},}
& \tikztableausmall{{X,X,2},{X,X,1},{\boldentry X,\boldentry 4},{\boldentry 2,\boldentry X},}
\\
 \uu_3\uu_2\uu_1\uu_2
& \uu_4\uu_3\uu_1\uu_2
& \uu_2\uu_4\uu_1\uu_2
\end{array}
\end{gather*}

The eighth term yields five:
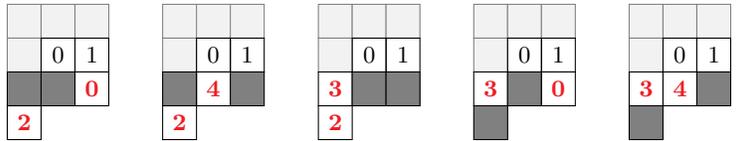
\begin{gather*}
\begin{array}{ccccc}
 \tikztableausmall{{X,X,X},{X,0,1},{\boldentry X,\boldentry X,\boldentry 0},{\boldentry 2},}
& \tikztableausmall{{X,X,X},{X,0,1},{\boldentry X,\boldentry 4,\boldentry X},{\boldentry 2},}
& \tikztableausmall{{X,X,X},{X,0,1},{\boldentry 3,\boldentry X,\boldentry X},{\boldentry 2},}
& \tikztableausmall{{X,X,X},{X,0,1},{\boldentry 3,\boldentry X,\boldentry 0},{\boldentry X},}
& \tikztableausmall{{X,X,X},{X,0,1},{\boldentry 3,\boldentry 4,\boldentry X},{\boldentry X},}
\\
 \uu_2\uu_0\uu_1\uu_0
& \uu_2\uu_4\uu_1\uu_0
& \uu_2\uu_3\uu_1\uu_0
& \uu_0\uu_3\uu_1\uu_0
& \uu_4\uu_3\uu_1\uu_0
\end{array}
\end{gather*}

The ninth term yields seven:
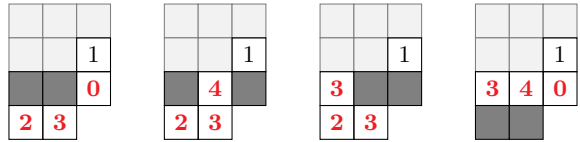
\begin{gather*}
\begin{array}{cccc}
\tikztableausmall{{X,X,X},{X,X,1},{\boldentry X,\boldentry X,\boldentry 0},{\boldentry 2,\boldentry 3},}
& \tikztableausmall{{X,X,X},{X,X,1},{\boldentry X,\boldentry 4,\boldentry X},{\boldentry 2,\boldentry 3},}
& \tikztableausmall{{X,X,X},{X,X,1},{\boldentry 3,\boldentry X,\boldentry X},{\boldentry 2,\boldentry 3},}
& \tikztableausmall{{X,X,X},{X,X,1},{\boldentry 3,\boldentry 4,\boldentry 0},{\boldentry X,\boldentry X},}
\\
 \uu_3\uu_2\uu_0\uu_1
& \uu_3\uu_2\uu_4\uu_1
& \uu_3\uu_2\uu_3\uu_1
& \uu_0\uu_4\uu_3\uu_1
\end{array}
\end{gather*}

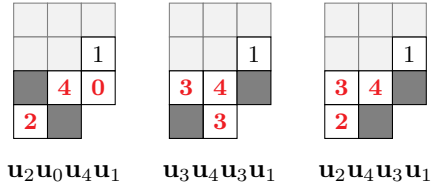
\begin{gather*}
\begin{array}{ccc}
 \tikztableausmall{{X,X,X},{X,X,1},{\boldentry X,\boldentry 4,\boldentry 0},{\boldentry 2,\boldentry X},}
& \tikztableausmall{{X,X,X},{X,X,1},{\boldentry 3,\boldentry 4,\boldentry X},{\boldentry X,\boldentry 3},}
& \tikztableausmall{{X,X,X},{X,X,1},{\boldentry 3,\boldentry 4,\boldentry X},{\boldentry 2,\boldentry X},}
\\
\uu_2\uu_0\uu_4\uu_1
& \uu_3\uu_4\uu_3\uu_1
& \uu_2\uu_4\uu_3\uu_1
\end{array}
\end{gather*}

The tenth term yields nine:
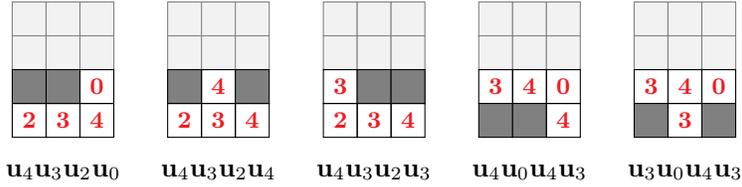
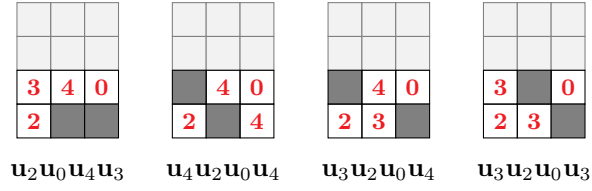
\begin{gather*}
\begin{array}{ccccc}
 \tikztableausmall{{X,X,X},{X,X,X},{\boldentry X,\boldentry X,\boldentry 0},{\boldentry 2,\boldentry 3,\boldentry 4},}
& \tikztableausmall{{X,X,X},{X,X,X},{\boldentry X,\boldentry 4,\boldentry X},{\boldentry 2,\boldentry 3,\boldentry 4},}
& \tikztableausmall{{X,X,X},{X,X,X},{\boldentry 3,\boldentry X,\boldentry X},{\boldentry 2,\boldentry 3,\boldentry 4},}
& \tikztableausmall{{X,X,X},{X,X,X},{\boldentry 3,\boldentry 4,\boldentry 0},{\boldentry X,\boldentry X,\boldentry 4},}
& \tikztableausmall{{X,X,X},{X,X,X},{\boldentry 3,\boldentry 4,\boldentry 0},{\boldentry X,\boldentry 3,\boldentry X},}
\\
 \uu_4\uu_3\uu_2\uu_0
& \uu_4\uu_3\uu_2\uu_4
& \uu_4\uu_3\uu_2\uu_3
& \uu_4\uu_0\uu_4\uu_3
& \uu_3\uu_0\uu_4\uu_3
\end{array}
\end{gather*}

\begin{gather*}
\begin{array}{ccccc}
 \tikztableausmall{{X,X,X},{X,X,X},{\boldentry 3,\boldentry 4,\boldentry 0},{\boldentry 2,\boldentry X,\boldentry X},}
& \tikztableausmall{{X,X,X},{X,X,X},{\boldentry X,\boldentry 4,\boldentry 0},{\boldentry 2,\boldentry X,\boldentry 4},}
& \tikztableausmall{{X,X,X},{X,X,X},{\boldentry X,\boldentry 4,\boldentry 0},{\boldentry 2,\boldentry 3,\boldentry X},}
& \tikztableausmall{{X,X,X},{X,X,X},{\boldentry 3,\boldentry X,\boldentry 0},{\boldentry 2,\boldentry 3,\boldentry X},}
\\
 \uu_2\uu_0\uu_4\uu_3
& \uu_4\uu_2\uu_0\uu_4
& \uu_3\uu_2\uu_0\uu_4
& \uu_3\uu_2\uu_0\uu_3
\end{array}
\end{gather*}

$\nckschur[4]_{3,1}$ is the sum of the $30$ terms above.
\end{Example}

\begin{Theorem}\label{rectangle-column}
Let $0 \leq i \leq k+1- c$. Then
\[
    \nckschur_{(c^{i}, (c-1)^{k+1-c-i})} =
    \sum_{\lambda \subseteq (c^{k+1-c})}
    \sum_{ \{ x_1, x_2, \dots, x_i \}}
        \uu_{(R \cup \lambda / \lambda) \setminus \{x_1, x_2, \dots, x_i\}},
\]
where the second sum is over all collections of $i$ cells in the rows below $R$, satisfying (C2) and no two of which lie in the same row.
\end{Theorem}

\begin{proof} This follows from Theorem \ref{thm:D_1^n}.
\end{proof}

\begin{Example}\label{eg:22}
    Let $R = (3,3)$ and $k=4$. Similar to Example \ref{eg:31}, $\nckschur[4]_{(2,2)}$ is the sum of all the monomials which occur after applying $D_{[1,1]}^{(3)}$ to the shapes starting in row $3$ in Example \ref{eg:33} (i.e. the red part). The first, second, fourth and sixth terms in Example \ref{eg:33} are zero after application of $D_{[1,1]}^{(3)}$. The third term yields one monomial and the fifth term yields two:
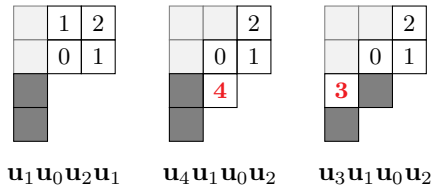
\begin{gather*}
\begin{array}{ccc}
 \tikztableausmall{{X,1,2},{X,0,1},{\boldentry X},{\boldentry X},}
& \tikztableausmall{{X,X,2},{X,0,1},{\boldentry X,\boldentry 4},{\boldentry X},}
& \tikztableausmall{{X,X,2},{X,0,1},{\boldentry 3,\boldentry X},{\boldentry X},}
\\
 \uu_1\uu_0\uu_2\uu_1
& \uu_4\uu_1\uu_0\uu_2
& \uu_3\uu_1\uu_0\uu_2
\end{array}
\end{gather*}

The seventh term yields three terms, as does the eighth:
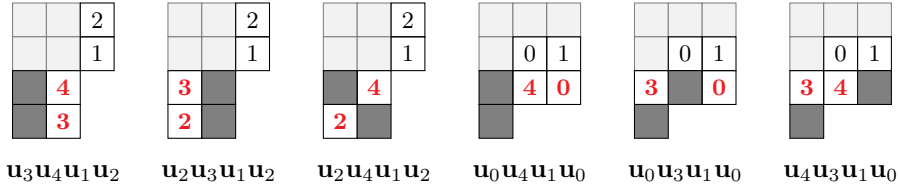
\begin{gather*}
\begin{array}{cccccc}
 \tikztableausmall{{X,X,2},{X,X,1},{\boldentry X,\boldentry 4},{\boldentry X,\boldentry 3},}
& \tikztableausmall{{X,X,2},{X,X,1},{\boldentry 3,\boldentry X},{\boldentry 2,\boldentry X},}
& \tikztableausmall{{X,X,2},{X,X,1},{\boldentry X,\boldentry 4},{\boldentry 2,\boldentry X},}
& \tikztableausmall{{X,X,X},{X,0,1},{\boldentry X,\boldentry 4,\boldentry 0},{\boldentry X},}
& \tikztableausmall{{X,X,X},{X,0,1},{\boldentry 3,\boldentry X,\boldentry 0},{\boldentry X},}
& \tikztableausmall{{X,X,X},{X,0,1},{\boldentry 3,\boldentry 4,\boldentry X},{\boldentry X},}
\\
 \uu_3\uu_4\uu_1\uu_2
& \uu_2\uu_3\uu_1\uu_2
& \uu_2\uu_4\uu_1\uu_2
& \uu_0\uu_4\uu_1\uu_0
& \uu_0\uu_3\uu_1\uu_0
& \uu_4\uu_3\uu_1\uu_0
\end{array}
\end{gather*}

The ninth term yields five terms:
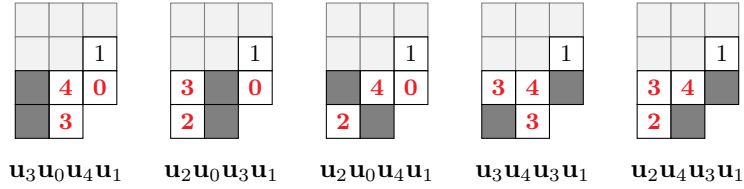
\begin{gather*}
\begin{array}{ccccc}
\tikztableausmall{{X,X,X},{X,X,1},{\boldentry X,\boldentry 4,\boldentry 0},{\boldentry X,\boldentry 3},}
& \tikztableausmall{{X,X,X},{X,X,1},{\boldentry 3,\boldentry X,\boldentry 0},{\boldentry 2,\boldentry X},}
& \tikztableausmall{{X,X,X},{X,X,1},{\boldentry X,\boldentry 4,\boldentry 0},{\boldentry 2,\boldentry X},}
& \tikztableausmall{{X,X,X},{X,X,1},{\boldentry 3,\boldentry 4,\boldentry X},{\boldentry X,\boldentry 3},}
& \tikztableausmall{{X,X,X},{X,X,1},{\boldentry 3,\boldentry 4,\boldentry X},{\boldentry 2,\boldentry X},}
\\
 \uu_3\uu_0\uu_4\uu_1
& \uu_2\uu_0\uu_3\uu_1
&\uu_2\uu_0\uu_4\uu_1
& \uu_3\uu_4\uu_3\uu_1
& \uu_2\uu_4\uu_3\uu_1
\end{array}
\end{gather*}

The tenth term yields six terms:
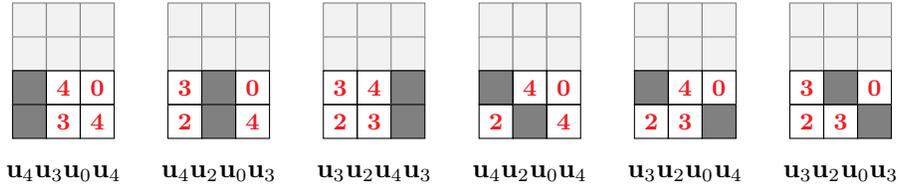
\begin{gather*}
\begin{array}{cccccc}
 \tikztableausmall{{X,X,X},{X,X,X},{\boldentry X,\boldentry 4,\boldentry 0},{\boldentry X,\boldentry 3,\boldentry 4},}
& \tikztableausmall{{X,X,X},{X,X,X},{\boldentry 3,\boldentry X,\boldentry 0},{\boldentry 2,\boldentry X,\boldentry 4},}
& \tikztableausmall{{X,X,X},{X,X,X},{\boldentry 3,\boldentry 4,\boldentry X},{\boldentry 2,\boldentry 3,\boldentry X},}
& \tikztableausmall{{X,X,X},{X,X,X},{\boldentry X,\boldentry 4,\boldentry 0},{\boldentry 2,\boldentry X,\boldentry 4},}
& \tikztableausmall{{X,X,X},{X,X,X},{\boldentry X,\boldentry 4,\boldentry 0},{\boldentry 2,\boldentry 3,\boldentry X},}
& \tikztableausmall{{X,X,X},{X,X,X},{\boldentry 3,\boldentry X,\boldentry 0},{\boldentry 2,\boldentry 3,\boldentry X},}
\\
 \uu_4\uu_3\uu_0\uu_4
& \uu_4\uu_2\uu_0\uu_3
& \uu_3\uu_2\uu_4\uu_3
& \uu_4\uu_2\uu_0\uu_4
& \uu_3\uu_2\uu_0\uu_4
& \uu_3\uu_2\uu_0\uu_3
\end{array}
\end{gather*}
$\nckschur[4]_{2,2}$ is the sum of the $20$ terms above.
\end{Example}
We end by using our result to compute a structure coefficient.
\begin{Example}
Let $k=4$. We will compute the product $s^{(4)}_{3,1}s^{(4)}_{2,1}$.
To compute this, we will use the action defined in \cite{LM05} on $(k+1)$-cores. The non-zero terms will be the coefficients from the expression in Example \ref{eg:31} which act non-trivially on the $5$-core $(2,1)$. Only four non-zero terms appear: \[\uu_3 \uu_2 \uu_0 \uu_3(2,1) = (4,2,1,1) \hspace{.3in} \uu_4 \uu_3 \uu_2 \uu_0 (2,1) = (5,2,2)\] \[\uu_3 \uu_1 \uu_0 \uu_2 (2,1) = (4,3,1) \hspace{.3in} \uu_4 \uu_3 \uu_2 \uu_3 (2,1) = (5,1,1,1).\]
Each of these terms has coefficient $1$ in the expansion from Example \ref{eg:31}. Using the bijection from \cite{LM05} between $5$-cores and $4$-bounded partitions gives us:
\[s^{(4)}_{3,1}s^{(4)}_{2,1} = 
s^{(4)}_{3, 2, 1, 1} + s^{(4)}_{3, 2, 2} + s^{(4)}_{3, 3, 1} + s^{(4)}_{4, 1, 1, 1}.\]
\end{Example}

\bibliographystyle{halpha}
\bibliography{references} 

\newcommand{\etalchar}[1]{$^{#1}$}
\begin{thebibliography}{BBTZ11}

\bibitem[BB05]{BB05}
Anders Bj{\"o}rner and Francesco Brenti.
\newblock {\em Combinatorics of {C}oxeter groups}, volume 231 of {\em Graduate
  Texts in Mathematics}.
\newblock Springer, New York, 2005.

\bibitem[BBTZ11]{BBTZ11}
C.~{Berg}, N.~{Bergeron}, H.~{Thomas}, and M.~{Zabrocki}.
\newblock {Expansion of $k$-Schur functions for maximal $k$-rectangles within
  the affine nilCoxeter algebra}.
\newblock {\em ArXiv e-prints}, July 2011, 1107.3610.

\bibitem[BBZ12]{BBZ}
Carolina Benedetti, Nantel Bergeron, and Mike Zabrocki.
\newblock Gromov-witten invariants for quasi-rectangles.
\newblock {\em In Preparation}, 2012.

\bibitem[BSS11]{BSS11}
C.~{Berg}, F.~{Saliola}, and L.~{Serrano}.
\newblock {The down operator and expansions of near rectangular k-Schur
  functions}.
\newblock {\em ArXiv e-prints}, December 2011, 1112.4460.

\bibitem[BSS12]{BSS12}
C.~Berg, F.~Saliola, and L.~Serrano.
\newblock Pieri operators on the affine nilcoxeter algebra.
\newblock {\em Trans. Amer. Math. Soc.}, (to appear), 2012.

\bibitem[Lam06]{Lam06}
Thomas Lam.
\newblock Affine {S}tanley symmetric functions.
\newblock {\em Amer. J. Math.}, 128(6):1553--1586, 2006.

\bibitem[Lam08]{Lam08}
Thomas Lam.
\newblock Schubert polynomials for the affine {G}rassmannian.
\newblock {\em J. Amer. Math. Soc.}, 21(1):259--281, 2008.

\bibitem[Lam10]{Lam10}
T.~Lam.
\newblock {Stanley symmetric functions and Peterson algebras}.
\newblock {\em ArXiv e-prints}, July 2010, 1007.2871.

\bibitem[LL12]{leung_gromov-witten_2012}
Naichung~Conan Leung and Changzheng Li.
\newblock {Gromov-Witten} invariants for {$G/B$} and {Pontryagin} product for
  {$\Omega K$}.
\newblock {\em Trans. Amer. Math. Soc.}, 364(05):2567--2599, January 2012.

\bibitem[LLM03]{LLM03}
L.~Lapointe, A.~Lascoux, and J.~Morse.
\newblock Tableau atoms and a new {M}acdonald positivity conjecture.
\newblock {\em Duke Math. J.}, 116(1):103--146, 2003.

\bibitem[LLMS10]{LLMS10}
Thomas Lam, Luc Lapointe, Jennifer Morse, and Mark Shimozono.
\newblock Affine insertion and {P}ieri rules for the affine {G}rassmannian.
\newblock {\em Mem. Amer. Math. Soc.}, 208(977):xii+82, 2010.

\bibitem[LM03]{LM03}
L.~Lapointe and J.~Morse.
\newblock Schur function analogs for a filtration of the symmetric function
  space.
\newblock {\em J. Combin. Theory Ser. A}, 101(2):191--224, 2003.

\bibitem[LM05]{LM05}
Luc Lapointe and Jennifer Morse.
\newblock Tableaux on {$k+1$}-cores, reduced words for affine permutations, and
  {$k$}-{S}chur expansions.
\newblock {\em J. Combin. Theory Ser. A}, 112(1):44--81, 2005.

\bibitem[LM07]{LM07}
Luc Lapointe and Jennifer Morse.
\newblock A {$k$}-tableau characterization of {$k$}-{S}chur functions.
\newblock {\em Adv. Math.}, 213(1):183--204, 2007.

\bibitem[LM08]{LM08}
Luc Lapointe and Jennifer Morse.
\newblock Quantum cohomology and the {$k$}-{S}chur basis.
\newblock {\em Trans. Amer. Math. Soc.}, 360(4):2021--2040, 2008.

\bibitem[LS07]{LS07}
Thomas~F. Lam and Mark Shimozono.
\newblock Dual graded graphs for {K}ac-{M}oody algebras.
\newblock {\em Algebra Number Theory}, 1(4):451--488, 2007.

\bibitem[LS10]{LSGW}
Thomas Lam and Mark Shimozono.
\newblock Quantum cohomology of {$G/P$} and homology of affine {G}rassmannian.
\newblock {\em Acta Math.}, 204(1):49--90, 2010.

\bibitem[MS12]{MS}
Jennifer Morse and Anne Schilling.
\newblock A combinatorial formula for fusion coefÞcients.
\newblock {\em DMTSC, to appear}, 2012.
\newblock FPSAC 2012.

\bibitem[S{\etalchar{+}}12]{sage}
W.\thinspace{}A. Stein et~al.
\newblock {\em {S}age {M}athematics {S}oftware ({V}ersion 4.7.2)}.
\newblock The Sage Development Team, 2012.
\newblock {\tt http://www.sagemath.org}.

\bibitem[SCc12]{sage-combinat}
The {S}age-{C}ombinat community.
\newblock {\em {{S}age-{C}ombinat}: enhancing Sage as a toolbox for computer
  exploration in algebraic combinatorics}.
\newblock The Sage Development Team, 2012.
\newblock {\tt http://combinat.sagemath.org}.

\end{thebibliography}
\end{document}